\documentclass[10pt,twoside]{article}

\usepackage{amsmath}
\usepackage{amsfonts}
\usepackage{amssymb}
\usepackage{amscd}
\usepackage{amsthm}
\usepackage{amsbsy}
\usepackage{graphicx}
\usepackage{bm}
\usepackage[linktoc=all]{hyperref}
\usepackage{pgf,tikz}
\usepackage{mathrsfs}
\usetikzlibrary{arrows}
\usepackage{makeidx}
\usepackage[nottoc]{tocbibind}

\def\@oddhead{\hfill \shorttitle \hfill \thepage}
\def\@evenhead{\thepage \hfill \shortauthor \hfill}
\def\@oddfoot{}
\def\@evenfoot{}

\theoremstyle{plain}
\newtheorem{proposition}{Proposition}[section]
\newtheorem{theorem}[proposition]{Theorem}
\newtheorem{conjecture}[proposition]{Conjecture}
\newtheorem{lemma}[proposition]{Lemma}
\newtheorem{corollary}[proposition]{Corollary}

\theoremstyle{definition}
\newtheorem{definition}[proposition]{Definition}

\theoremstyle{remark}
\newtheorem{remark}[proposition]{Remark}
\newtheorem{example}[proposition]{Example}
\newtheorem{question}[proposition]{Question}
\newtheorem{examples}[proposition]{Examples}

\everymath{\displaystyle}

\DeclareMathOperator{\Conv}{Conv}
\DeclareMathOperator{\Isom}{Isom}
\DeclareMathOperator{\Id}{Id}
\DeclareMathOperator{\rank}{rank}
\DeclareMathOperator{\Lk}{Lk}
\DeclareMathOperator{\Or}{O}
\newcommand{\R}{\mathbb{R}}                          % |R
\newcommand{\C}{\mathbb{C}}                          % |C
\newcommand{\N}{\mathbb{N}}                          % |N
\newcommand{\Z}{\mathbb{Z}}  
\newcommand{\bo}{\partial}  
\DeclareMathOperator{\GL}{GL}
\DeclareMathOperator{\SL}{SL}
\DeclareMathOperator{\SU}{SU}
\DeclareMathOperator{\Sp}{Sp}
\DeclareMathOperator{\F}{F}
\DeclareMathOperator{\PSL}{PSL}
\DeclareMathOperator{\SO}{SO}
\DeclareMathOperator{\Prob}{Prob}
\DeclareMathOperator{\Aut}{Aut}

\newcommand{\action}{\curvearrowright}  
\makeindex
\date{}
\title{\ \\[0.4cm] \ \\ \bf  Groups acting on spaces of non-positive curvature}
\author{Bruno Duchesne\footnote{Institut \'Elie Cartan de Lorraine, Universit\'e de Lorraine, B.P. 70239, F-54506 Vandoeuvre-l\`es-Nancy Cedex, France. E-mail:  bruno.duchesne@univ-lorraine.fr. Supported in part by Projet ANR-14-CE25-0004 GAMME and R\'egion Lorraine.}\hspace{2mm}}
\begin{document}
%-------------------

\maketitle

%------------------------

\thispagestyle{empty}

%--------------------------------------
\begin{abstract}
\vskip 3mm\footnotesize{
\vskip 4.5mm

\noindent
In this survey article, we present some panorama of groups acting on metric spaces of non-positive curvature. We introduce the main examples and their rigidity properties, we show the links between algebraic or analytic properties of the group and geometric properties of the space. We conclude with a few conjectures in the subject.

\vspace*{2mm} \noindent{\bf 2000 Mathematics Subject Classification: }20F65, 53C23, 20F67.

\vspace*{2mm} \noindent{\bf Keywords and Phrases: } Group actions, spaces of non-positive curvature, CAT(0) spaces, rigidity, superrigidity, amenability, Haagerup property, rank rigidity conjecture, flat closing conjecture.}
\end{abstract}
%\vspace*{-9.2cm}
%%%%%%%%%%%%%%%%%%%%%%%%%%%%%%%%%%%%%
%\vspace*{2mm} \noindent\hspace*{82mm}
%\begin{picture}(41,10)(0,0)\thicklines\setlength{\unitlength}{1mm}
%\put(0,2){\line(1,0){41}} \put(0,16){\line(1,0){41}}%
%\put(0,12.){\sl \copyright\hspace{1mm}Higher Education Press}
%\put(0,7.8){\sl \hspace*{4.8mm}and International Press}%
%\put(0,3.6){{\sl \hspace*{4.8mm}Beijing-Boston} }
%\end{picture}
%
%\vspace*{-17.6mm}\noindent{{\sl Handbook of\\
%Group Actions (Vol. III)}\\ALM\,??, pp.\,?--?} \vskip8mm
%%%%%%%%%%%%%%%%%%
%\vspace*{7.6cm}
% ----------------------------------------------------------------
\tableofcontents
\section{Introduction}

According to the so-called Erlangen Program of  Felix Klein \cite{Klein:2008xg,MR3379703}, a \emph{geometry} is a set and a group that preserves some invariants. On the other hand, groups are actors and they are better understood via the study of their actions. In this paper, we aim to present the strong links between algebraic or analytic properties of groups and geometric properties of spaces of non-positive curvature on which the groups act. This is actually a part of \emph{Geometric Group Theory} and we recommend the survey \cite{MR1886672} for a panoramic view on that larger subject.

The metric spaces of non-positive curvature we have in mind are the so-called CAT(0) spaces. These are spaces with a metric condition meaning that the space is at least as non-positively curved as the Euclidean plane. The definition goes back to Alexandrov and Busemann in the 1950's and was popularised by Gromov who coined the name CAT(0) after Cartan, Alexandrov and Toponogov.

There are very good books about spaces of non-positive curvature. The current quite comprehensive reference is \cite{MR1744486}. The books \cite{MR1377265} and \cite{MR823981} are a bit older but still interesting and useful. For the study of convexity in non-positive curvature beyond CAT(0) spaces, \cite{MR3156529} is a good reference.

This paper is not another introduction to the subject but the aim is to guide the reader toward the main examples, some important results in the field and directions of current research. 

We  present the leading examples of CAT(0) spaces and groups acting on them: symmetric spaces, Euclidean buildings and CAT(0) cube complexes. After the basics about CAT(0) spaces and the description of these examples, we aim to give a flavour of the subject by treating a few topics.  The first topic deals with the question of generality/rigidity of the spaces associated to linear groups (that are symmetric spaces and Euclidean buildings). The second one deals with some analytic properties of groups: amenability, the Haagerup property, property (T) and amenability at infinity. Finally we discuss two conjectures: the rank rigidity conjecture and the flat closing conjecture. 

  The choice of topics is not exhaustive and is definitely subjective. We apologise for the important facts we may have missed. We do not enter into the details of each topic but we encourage the reader to continue his reading with references in the quite long bibliography. 
  
\section{Spaces of non-positive curvature}
\subsection{CAT(0) spaces}
A metric space $(X,d)$ is \emph{geodesic}\index{space!geodesic}\index{geodesic!space} if there is a geodesic segment joining any two points: For any $x,y\in X$, there exists $\gamma\colon[0,d(x,y)]\to X$ such that for any $t,t'\in[0,d(x,y)]$, $d(\gamma(t),\gamma(t'))=|t-t'|$, $\gamma(0)=x$ and $\gamma(d(x,y))=y$. A priori, geodesic segments between two points are not unique. For two points $x,y$,  a \emph{midpoint}\index{midpoint}  for $x$ and $y$ is a point $\gamma(d(x,y)/2)$ for some geodesic segment $\gamma$ between $x$ and $y$. Let us now define CAT(0) spaces via the so-called CAT(0) inequality
\begin{definition} A complete geodesic metric space is \emph{CAT(0)}\index{CAT(0)!space}\index{space!CAT(0)} if for any triple of points $x,y,z$ and a midpoint $m$ between $y$ and $z$, one has 

\begin{equation}\label{CAT(0)}d(x,m)^2\leq \frac{1}{2}(d(x,y)^2+d(x,z)^2)-\frac{1}{4}d(y,z)^2.\index{CAT(0)!inequality}\index{inequality!CAT(0)}\end{equation}
\end{definition}

\begin{remark}We use completeness and geodesicity in the definition of CAT(0) spaces. This is not a completely standard choice but with this definition, the metric completion of a metric space $X$ satisfying the following list of hypotheses is automatically CAT(0).
\begin{itemize}
\item For any two points $x,y\in X$, there is $m\in X$ with $d(x,m)=d(y,m)=\frac{1}{2}d(x,y)$,
\item for any triple $x,y,z\in X$ and $m\in X$ such that $d(x,m)=d(y,m)=\frac{1}{2}d(x,y)$,  Inequality \eqref{CAT(0)} holds.
\end{itemize}
\end{remark}

For a Euclidean space, the parallelogram law yields equality in Inequality \eqref{CAT(0)}. So, Inequality \eqref{CAT(0)}, called the \emph{CAT(0) inequality}, means that $X$ is as least as non-positively curved as a Euclidean plane. There is another equivalent way to express the same idea of non-positive curvature\index{curvature!non-positive}\index{non-positive curvature}: geodesic triangles are thinner than Euclidean ones (see \cite[Proposition II.1.7]{MR1744486}).

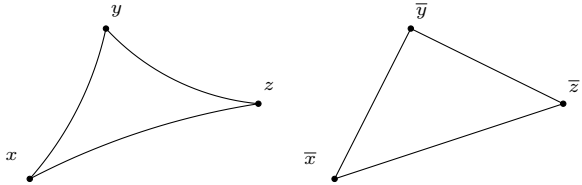
\begin{figure}
\begin{center}
\begin{tikzpicture}[line cap=round,line join=round,>=triangle 45,x=1.0cm,y=1.0cm]
\clip(1.34,0.4) rectangle (9.72,3.58);
\draw (7.,3.)-- (6.,1.);
\draw (6.,1.)-- (9.,2.);
\draw (7.,3.)-- (9.,2.);
\draw [shift={(-1.364,3.932)}] plot[domain=5.566294760392265:6.072780635965295,variable=\t]({1.*4.462411903892333*cos(\t r)+0.*4.462411903892333*sin(\t r)},{0.*4.462411903892333*cos(\t r)+1.*4.462411903892333*sin(\t r)});
\draw [shift={(6.484,-7.452)}] plot[domain=1.7265288277733808:2.058564934609697,variable=\t]({1.*9.56778762306104*cos(\t r)+0.*9.56778762306104*sin(\t r)},{0.*9.56778762306104*cos(\t r)+1.*9.56778762306104*sin(\t r)});
\draw [shift={(5.292,5.084)}] plot[domain=3.8794945720440968:4.61798817072367,variable=\t]({1.*3.0977927625972654*cos(\t r)+0.*3.0977927625972654*sin(\t r)},{0.*3.0977927625972654*cos(\t r)+1.*3.0977927625972654*sin(\t r)});
\begin{scriptsize}
\draw [fill=black] (2.,1.) circle (1.0pt);
\draw[color=black] (1.76,1.3) node {$x$};
\draw [fill=black] (3.,3.) circle (1.0pt);
\draw[color=black] (3.14,3.24) node {$y$};
\draw [fill=black] (5.,2.) circle (1.0pt);
\draw[color=black] (5.14,2.24) node {$z$};
\draw [fill=black] (6.,1.) circle (1.0pt);
\draw[color=black] (5.68,1.28) node {$\overline x$};
\draw [fill=black] (7.,3.) circle (1.0pt);
\draw[color=black] (7.14,3.24) node {$\overline y$};
\draw [fill=black] (9.,2.) circle (1.0pt);
\draw[color=black] (9.14,2.24) node {$\overline z$};
\end{scriptsize}
\end{tikzpicture}
\caption{A geodesic triangle in a CAT(0) space is thinner than a triangle in the Euclidean plane with the same side lengths.}
\end{center}
\end{figure}

It is a simple consequence of the CAT(0) inequality that two points $x,y$ are linked by a unique geodesic segment that we denote by $[x,y]$. Moreover it is easy to see that a CAT(0) space $X$ is contractible. Fix a point $x_0\in X$ and define $h(x,t)$ to be the point $tx_0+(1-t)x$ in an isometric parametrisation of the segment $[x,x_0]$ where $t\in[0,1]$. The CAT(0) inequality (actually convexity is sufficient) shows that $h$ is continuous and is a retraction from $X$ to $\{x_0\}$. 

A metric space is \emph{locally CAT(0)}\index{space!locally CAT(0)}\index{locally CAT(0) space} if any point is the center of a ball of positive radius which is CAT(0). Simple connectedness and the local CAT(0) condition is sufficient to recover the CAT(0) condition globally.

\begin{theorem}[Cartan-Hadamard theorem]\index{theorem!Cartan-Hadamard}\index{Cartan-Hadamard theorem} A complete metric space is CAT(0) if and only if it is locally CAT(0) and simply connected.
\end{theorem}

\begin{examples}The following are examples of CAT(0) spaces.

\begin{enumerate}
\item Euclidean and more generally Hilbert spaces.
\item Simplicial trees with the length metric and more generally real trees. 
\item Complete simply connected Riemannian manifolds of non-positive sectional curvature. Such manifolds are called \emph{Hadamard manifolds}\index{Hadamard manifold}.
\end{enumerate}
\end{examples}

Let us comment on these examples. It is clear that Hilbert spaces are CAT(0) spaces because of the parallelogram law. For Banach spaces, the CAT(0) condition actually implies  the parallelogram law. From that point, one can recover that the norm comes from a scalar product and thus the only Banach spaces that are CAT(0) are actually Hilbert spaces.

For Riemannian manifolds, non-positivity of the sectional curvature is equivalent to the CAT(0) inequality for small balls. This follows from the seminal works of Busemann and Alexandrov \cite[II.1.17]{MR1744486}. The Cartan-Hadamard theorem allows to go from local to global.

Let us see a first simple consequence of the CAT(0) condition for a group acting by isometries on a CAT(0) space.
\begin{theorem}[Cartan fixed point theorem]\index{Cartan fixed point theorem}\index{theorem!Cartan fixed point} Let $G$ be a group acting by isometries on a CAT(0) space. If $G$ has a bounded orbit then $G$ has a fixed point.
\end{theorem}

The proof of the Cartan fixed point theorem relies on the fact that any bounded subset $B$ of a CAT(0) space has a \emph{circumcenter}\index{circumcenter}. This is the center of the unique closed ball of minimal radius containing $B$. In particular, a group with a bounded orbit fixes the circumcenter of this orbit.

This gives a proof of the fact that any compact subgroup of $\GL_n(\R)$ lies in the orthogonal group of some scalar product of $\R^n$. The set of scalar products on $\R^n$ may be endowed with a CAT(0) metric such that $\GL_n(\R)$ acts by isometries on it.
\begin{definition}Let $X$ be a CAT(0) space. A subspace of $X$ is a \emph{flat}\index{flat!subspace}\index{subspace!flat} if it is isometric to some Euclidean space. The \emph{rank}\index{rank} of $X$ is the supremum of dimensions of flats in $X$.
\end{definition}
Flat subsets, that is, subsets isometric to some Euclidean space, in CAT(0) spaces are limit cases of non-positive curved spaces and thus play a particular role. The rank of a CAT(0) space is an important invariant. For examples, it appears in two conjectures discussed in Section \ref{rrc}.

Isometries of CAT(0) spaces may have very different dynamical behaviours.
\begin{definition}Let $X$ be a CAT(0) space and $g\in\Isom(X)$. The \emph{translation length}\index{translation length}\index{length!translation} of $g$ is $\inf_{x\in X} d(gx,x)$. If this infimum is achieved, $g$ is said to be \emph{semi-simple}\index{semi-simple isometry}\index{isometry!semi-simple} and \emph{parabolic}\index{parabolic isometry}\index{isometry!parabolic} otherwise.
\end{definition}

In a CAT(0) space, the semi-simple isometries are the ones we understand. The parabolic ones are more or less mysterious. If $g$ is semi-simple with vanishing translation length then $g$ has fixed points. If $g$ is semi-simple with positive translation length then there is a geodesic line which is $g$-invariant and $g$ acts as a translation on this geodesic line. 

\subsection{Convexity and weaker notions of non-positive curvature}
The CAT(0) inequality implies strong convexity consequences that we briefly describe.\index{convexity}

\begin{definition} Let $(X,d)$ be a geodesic metric space. A function $f\colon X\to \R$ is \emph{convex}\index{convex!function}\index{function!convex} if for any geodesic segment $[x,y]\subset X$, the restriction of $f$ to $[x,y]$ is convex (as a function on a real interval).
\end{definition}

In CAT(0) spaces, the distance to a point is a convex function. More generally, for two parametrisations with constant speeds $\gamma,\gamma'\colon[0,1]\to X$ of geodesic segments, the distance function $t\mapsto d(\gamma(t),\gamma'(t))$ is convex. 

A subspace of a CAT(0) space is \emph{convex}\index{convex!subspace}\index{subspace!convex} if it contains the geodesic segment between any two of its points. For a closed convex subset $C$, any point $x\in X$ has a unique point $p_C(x)\in C$ minimising the distance to $x$. The map $p_C\colon X\to C$ is 1-Lipschitz and is called the \emph{projection}\index{projection!on a subspace} onto $C$. Let $d_C$ be the \emph{distance function}\index{distance function}\index{function!distance} to $C$: for any $x\in C$, $d_C(x)=d(x,p_C(x))$. This distance function $d_C$ is also convex.

We have seen that the CAT(0) condition gives a notion of non-positive curvature for metric spaces. Comparing to Banach spaces, one may see the CAT(0) inequality as a (very strong) uniform convexity notion. Concentrating on the convexity of the distance, let us introduce two weaker notions of non-positive curvature for metric spaces.

\begin{definition} Let $(X,d)$ be a geodesic metric space. The space $X$ is \emph{Busemann non-positively curved}\index{Busemann non-positively curved space}\index{space!Busemann non-positively curved} if for any geodesic triangle $[x,y],[y,y']$ and $[y',x]$, the midpoints $m$ of $[x,y]$ and $m'$ of $[x,y']$, one has

$$d(m,m')\leq\frac{1}{2}d(y,y').$$
\end{definition}

This definition has different equivalent formulations (see \cite[Proposition 8.1.2]{MR2132506}) and it implies uniqueness of geodesic segments between two points. For example, a normed space is Busemann non-positively curved if and only if it is strictly convex.

A normed space which is not strictly convex may have many geodesics between two points (think to $\R^2$ with the sup norm) but some geodesics are distinguished: the affine geodesics. Actually, if one forgets the linear structure and remembers only the metrics one can recover the linear structure. The famous Mazur-Ulam theorem tells us that the isometries are affine because isometries remembers affine midpoints. In particular, affine geodesics are very special because they are preserved by isometries. With this idea that some specific geodesics may be more important than others, one has the following even weaker notion of non-positive curvature.

\begin{definition}Let $(X,d)$ be a metric space. A \emph{geodesic bicombing}\index{geodesic bicombing} is a choice of a geodesic segment $\sigma_{x,y}$ between $x$ and $y$ in $X$ for any two points $x,y\in X$. That is $\sigma_{x,y}\colon [0,1]\to X$ and
\begin{itemize}
\item $\sigma_{x,y}(0)=x$,
\item $\sigma_{x,y}(1)=y$,
\item for any $t,t'\in[0,1]$, $d(\sigma_{x,y}(t),\sigma_{x,y}(t'))=|t-t'|d(x,y)$.
\end{itemize}
A geodesic bicombing is said to be \emph{convex}\index{convex!geodesic bicombing} if for any quadruple $x,y,x',y'$,

$$t\mapsto d(\sigma_{x,y}(t),\sigma_{x',y'}(t))$$
is convex.
\end{definition}

For example, in any normed space, the geodesic bicombing given by affine segments is convex.

\begin{example}
The space $\GL_n(\R)/\Or_n(\R)$ is a symmetric space of non-compact type (see \S \ref{syms}) and thus a CAT(0) space. Leaving finite dimension, the natural analog in infinite dimension is $\GL(\mathcal{H})/\Or(\mathcal{H})$ where $\GL(\mathcal{H})$ is the group of all bounded invertible operators of the Hilbert space $\mathcal{H}$ and $\Or(\mathcal{H})$ is the orthogonal group. This space is no more CAT(0) but there is a convex geodesic bicombing. This bicombing comes from the fact that $\GL(\mathcal{H})/\Or(\mathcal{H})$ is a manifold of infinite dimension that can be identified with the space of symmetric definite positive operators. The tangent space at $\Id$ is the space of symmetric bounded operators. This vector space with the operator norm is a Banach space and the exponential map sends linear geodesics to geodesics belonging to the geodesic bicombing. The action of $\GL(\mathcal{H})$ is isometric and moreover preserves the geodesic bicombing.
\end{example}

%\begin{remark} A metric space $(X,d)$ with a geodesic bicombing is contractible. The homotopy $H$ to a point $x_0\in X$ is given by the following map
%
%$$\begin{array}{rcl}
%X\times[0,1]&\to& X\\
%(x,t)&\mapsto&\sigma_{x,x_0}.
%\end{array}$$
%
%It is easy to show that $$d(H(x,t),H(y,t'))\leq\max(t,t')d(x,y)+|t-t'|\max(d(x_0,x),d(x_0,y))$$ and thus $H$ is continuous. 
%
%In particular, actions of groups on such spaces may lead to computations of some cohomology groups.
%
%In the other direction, to check that a space is CAT(0) (or more generally Busemann non-positively curved) it is sufficient to check the non-positive curvature condition locally and simply connectedness. This is the \emph{Cartan-Hadamard theorem} and its generalizations \cite{MR1744486,MR2132506}.
%\end{remark}

\subsection{The proper case}

The CAT(0) inequality (and more generally convex bicombings) yields a notion of non-positive curvature beyond manifolds. In particular, CAT(0) spaces may be not locally-compact. However, locally compact CAT(0) spaces have more enjoyable properties. 

\begin{definition} A metric space is said to be \emph{proper}\index{proper!space}\index{space!proper} if closed balls are compact.\end{definition}
For CAT(0) spaces, properness is equivalent to local compactness. For proper CAT(0) spaces, there is a natural and pleasant topology on the isometry group.

\begin{proposition} Let $(X,d)$ be a proper CAT(0) space. Its isometry group endowed with the compact-open topology is a locally compact second countable group.
\end{proposition}
In particular, such groups have Haar measures and may have lattices (that are discrete subgroups of finite covolume).

Let $G$ be a topological group acting by isometries on a metric space $(X,d)$. The action is \emph{continuous}\index{action!continuous}\index{continuous action} if the map $(\gamma,x)\mapsto \gamma x$ from $\Gamma\times X$ to $X$ is continuous. It is \emph{proper}\index{proper!action}\index{action!proper} if for any compact subsets $K_1,K_2\subset X$, $\{g\in G|\ gK_1\cap K_2\neq\emptyset\}$ is compact. In case $X$ is proper and $G$ is discrete, the action is proper if and only if it is properly discontinuous, that is for any $x,y\in X$, there are neighbourhoods $U_x,U_y$ of $x$ and $y$ such that $\{g\in,\ gU_x\cap U_y\neq\emptyset\}$ is finite. The action is \emph{cocompact}\index{action!cocompact}\index{cocompact action} if the topological quotient $G\backslash X$ is compact.

The interplay between the algebraic properties of a group acting on a CAT(0) space and the geometry of the space is expected to be richer when the action is proper and cocompact. This idea leads to the following definition.

\begin{definition} Let $G$ be a group acting by isometries on a CAT(0) space $X$. The action is \emph{geometric}\index{geometric action}\index{action!geometric} (or the group acts \emph{geometrically}) if it is proper and cocompact. A countable group $\Gamma$ is \emph{CAT(0)}\index{CAT(0)!group}\index{group!CAT(0)} if acts geometrically on a proper CAT(0) space. \end{definition}

\subsection{Boundary at infinity}
There is a nice geometric object associated to any CAT(0) space $X$: its \emph{boundary at infinity}\index{boundary!at infinity}. A \emph{geodesic ray} \index{geodesic!ray}\index{ray!geodesic} is the image of some isometric map $[0,+\infty)\to X$.

\begin{definition} Let $X$ be CAT(0) space. The boundary at infinity $\partial X$ of $X$ is the set of all geodesic rays (with any base point) where two of them are identified if and only if there are at bounded Hausdorff distance  from one another.
\end{definition}

If one fixes a base point $x_0\in X$, there is exactly one geodesic ray starting from $x_0$ in each class and thus $\partial X$ can be identified with the set of geodesic rays starting from $x_0$. This boundary is sometimes called the \emph{visual boundary}\index{boundary!visual} because it corresponds to the vision of infinity from $x_0$. There is a natural way to topologize $\overline{X}=X\cup\bo X$. Fix a point $x_0\in X$. One can identify $\overline{X}$ with the inverse limit of the system of balls of radius $r>0$ around $x_0$. In this way $X$ embeds as an open set in $\overline X$. A sequence of points $x_n\in\overline X$ converges to some point $\xi\in\bo X$ if for all $r>0$ the point $x_n(r)$ at distance  $r$ from $x_0$ on $[x_0,x_n]$ converges to the point $x(r)\in[x_0,\xi)$ at distance $r$ from $x_0$.
\begin{center}
\begin{tikzpicture}[line cap=round,line join=round,>=triangle 45,x=1.0cm,y=1.0cm]
\clip(-1.04,-0.58) rectangle (6.18,2.8);
\draw (0.,0.)-- (2.,2.);
\draw (0.,0.)-- (5.34,2.3);
\draw [shift={(0.,0.)}] plot[domain=0.20431586460012127:1.0026915943204695,variable=\t]({1.*2.2671568097509267*cos(\t r)+0.*2.2671568097509267*sin(\t r)},{0.*2.2671568097509267*cos(\t r)+1.*2.2671568097509267*sin(\t r)});
\begin{scriptsize}
\draw [fill=black] (0.,0.) circle (1.0pt);
\draw[color=black] (-0.28,0.42) node {$x_0$};
\draw [fill=black] (2.,2.) circle (1.0pt);
\draw[color=black] (2.18,2.3) node {$x_n$};
%\draw [fill=black] (5.34,2.3) circle (1.0pt);
\draw[color=black] (5.6,2.54) node {$\xi$};
\draw [fill=black] (1.6031219541881399,1.6031219541881399) circle (1.0pt);
\draw[color=black] (1.,1.72) node {$x_n(r)$};
\draw [fill=black] (2.0822290073879253,0.8968402091745745) circle (1.0pt);
\draw[color=black] (2.66,0.84) node {$x(r)$};
\end{scriptsize}
\end{tikzpicture}
\end{center}

For any simply connected complete Riemannian manifold $X$ of non-positive curvature and dimension $n$, the boundary $\bo X$ is homeomorphic to a sphere of dimension $n-1$. In particular, for Euclidean spaces the boundary at infinity is a Euclidean sphere. This statement can be made metrical endowing this boundary with the angular metric 
\cite[II.9.4]{MR1744486}.

The construction of this boundary is functorial: if a group acts by isometries on a CAT(0) space it also acts by homeomorphisms on its boundary.

This boundary is particularly convenient for proper CAT(0) spaces.
\begin{proposition} If $X$ is a proper CAT(0) space then $\overline X$ is compact.
\end{proposition}
For example, this proposition shows that any parabolic isometry $g$ of a proper CAT(0) space has a fixed point at infinity. Choose a minimising sequence for the function $x\mapsto d(gx,x)$. This sequence has a subsequence that converges to a point at infinity which is fixed. 
\subsection{Some simple geometric questions}
Despite the fact that CAT(0) spaces became more and more popular in the last decades, some elementary (in their statement) questions are still completely open. Let us state and comment two of them. These questions are particular cases of the following theme: \textit{what results of geometric functional analysis hold for metric spaces of non-positive curvature?} By geometric functional analysis we mean the study of metric and probabilistic questions or questions about convexity for Banach spaces or more generally Fr\'echet spaces. Fr\'echet spaces are locally convex vector spaces metrizable by a complete and translation invariant metric. 

Let $X$ be a CAT(0) space and $Y$ a subset of $X$. The \emph{closed convex hull}\index{convex!hull} of $Y$ is the smallest closed and convex subset of $X$ that contains $Y$. This is the intersection of all closed and convex subspaces of $X$ containing $Y$. We denote it by $\overline{\Conv}(Y)$.
\begin{question}\label{convexhull}Let $X$ be a CAT(0) space. Is the closed convex hull of a compact subset still compact?
\end{question}

For convenience, let us say that the space $X$ has \emph{Property (CH)}\index{property!(CH)} if the closed convex hull of any compact subset of $X$ is compact. This question appeared in \cite[Section 4]{MR1385525} (where the name of the property is coined) and is discussed in \cite{6627}. Observe it is not difficult to find compacta with non-closed convex hulls. 

If the CAT(0) space is \emph{proper} then the answer is obviously positive because any compact subset is contained in some closed ball and thus the closed convex hull, which is included in this ball, is compact. The following lemma shows that it is sufficient to answer Question \ref{convexhull} for finite subsets.

\begin{lemma}\label{finite} Let $X$ be a CAT(0) space such that the closed convex hull of any finite number of points is compact. Then $X$ has Property (CH).
\end{lemma}

\begin{proof} Let $K$ be a compact subset of $X$. We will show that for any sequence of $\overline{\Conv}(K)$ and $\varepsilon>0$, one can extract a subsequence of diameter less than $\varepsilon$. A diagonal extraction yields a Cauchy subsequence and thanks to completeness, this subsequence converges.

So, let $K$ be a compact subspace, $\varepsilon>0$ and $(k_i)$ be a sequence of $\overline{\Conv}(K)$, the closed convex hull of $K$. There is a finite number of points $x_1,\dots,x_n$ such that $K\subset \cup_{j=1}^nB(x_j,\varepsilon/3)$. Let $C$ be $\overline{\Conv}(x_1,\dots,x_n)$ which is compact. By convexity of the distance function $d_C$, $\overline{\Conv}(K)$is contained in the $\varepsilon/3$-neighborhood of $C$. In particular, for any $i\in\N$ there is an $l_i\in C$ such that $d(k_i,l_i)<\varepsilon/3$. By compactness of $C$, one can extract a subsequence of $(l_{\varphi(i)})$ of diameter less than $\varepsilon/3$. By the triangle inequality, $(k_{\varphi(i)})$ is a subsequence of $(k_i)$ of diameter less than $\varepsilon$.
\end{proof}

For two points the answer to Question \ref{convexhull} is clearly positive since the convex hull of two points is merely the segment between them. Actually the same question for three points is open and seems to be as hard as the original question.

\begin{question}[{\cite[6.B$_1$(f)]{MR1253544}}]Is the closed convex hull of any three points in a CAT(0) space compact?
\end{question}

Thanks to Lemma \ref{finite} and the fact that a finite number of points in a Hilbert space lies in a finite-dimensional subspace, one knows that Hilbert spaces have Property (CH). Actually the question holds also for Busemann non-positively curved space and is also open in this context (to our knowledge). For linear convexity, the result holds true for Banach spaces, see \cite[Theorem 2.10]{MR2338582}, and relies on an analog of Lemma \ref{finite}.

We do not have a plethora of well understood examples  of non proper CAT(0) spaces. One class of such examples is given by infinite-dimensional symmetric spaces of non-positive curvature (see \S \ref{syms}). In \cite[Section 4]{MR1385525}, it is stated without proof that infinite-dimensional symmetric spaces have property (CH). Let us give a precise statement and a proof.

\begin{proposition}\label{convexsymmetric} Simply connected Hilbertian symmetric spaces with non-positive curvature operator have Property (CH).
\end{proposition}

\begin{proof}It follows from the classification of Hilbertian symmetric spaces \cite[Theorem 1.8]{BD15} that any such symmetric space is obtained as the closure of an increasing union of finite-dimensional symmetric spaces. In particular, for any $\varepsilon>0$ any compact subspace $K$ lies in the $\varepsilon$-neighbourhood of a finite number of points lying in a finite-dimensional subspace. As in Lemma \ref{finite}, $K$ lies in the $\varepsilon$-neighbourhood of a convex compact space and thus its closed convex hull is compact.
\end{proof}

\begin{remark} Proposition \ref{convexsymmetric} is not straightforward because a finite number of points (even three points) are not necessarily included in some finite-dimensional totally geodesic subspace.  
\end{remark}

\begin{definition} Let $Y$ be a convex subspace of a CAT(0) space. An \emph{extremal point}\index{extremal point}\index{point!extremal} of $Y$ is a point $y\in Y$ such that if $x,x'\in Y$ and $y\in[x,x']$ then $x=y$ or $x'=y$.
\end{definition}

Observe that in this definition, it suffices to consider points $x,x'$ such that $y$ is the midpoint of $[x,x']$. The classical Krein-Milman theorem asserts that a compact convex subset $C$ of a locally convex topological vector space has extremal points and actually this is equivalent (via the Hahn-Banach theorem) to the fact that $C$ is the closed convex hull of its extremal points. 

If one thinks of the geometry of Banach spaces as a leading source of intuition for the geometry of metric spaces of non-positive curvature, it is natural to ask if the same is true for metric spaces with a convex bicombing. In that case the answer is positive, as it was shown in \cite{Buehler:yu}. Moreover, a convex compact subset is the closed convex hull of its extremal points.

Hilbert spaces are the only Fr\'echet spaces which are CAT(0) (for any complete, invariant by translations compatible distance). In that case, closed convex bounded subspaces are compact for the weak topology and the Krein-Milman theorem for the weak topology implies that closed convex bounded subsets are the convex hull of their extremal points. So it is natural to ask the following question.

\begin{question}\label{extremal}Let $Y$ be a non-empty closed convex bounded subspace of a CAT(0) space. Is it true that $X$ has extremal points?
\end{question}

This question was asked by Nicolas Monod in the Master class "Geometric and Cohomologic Group Theory: Rigidity and Deformation" held at CIRM in 2006. Let us give a very short proof in the compact CAT(0) case.

\begin{proposition} Let $X$ be a (non-empty) compact CAT(0) space. Then $X$ has extremal points and $X$ is the closed convex hull of its extremal points.
\end{proposition}

\begin{proof} Let $x_0\in X$ and using compactness, choose $x\in X$ such that $d(x_0,x)=\max_{y\in X}d(x_0,y)$. Let $y,z\in X$ such that $x$ is the midpoint of $[y,z]$. By the CAT(0) inequality,
$$d(x_0,x)^2\leq 1/2(d(x_0,y)^2+d(x_0,z)^2)-1/4d(y,z)^2\leq d(x_0,x)^2-1/4d(y,z)^2$$
and thus $y=z=x$.
Once one knows there are extremal points, an argument of the same spirit shows that $X$ is the convex hull of its extremal points \cite{Buehler:yu}.
\end{proof}

Of course, it is also natural to ask if the full result of the Krein-Milman theorem holds for CAT(0) spaces.

\begin{question}\label{KMi}Is it true that a closed convex bounded subset of a CAT(0) space is the closure of the convex hull of its extremal points?\end{question}
Questions \ref{KMi} and \ref{extremal} were answered negatively during the writing of this paper. Nicolas Monod gave a counterexample to Question \ref{extremal} called \emph{the rose}. The rose of Monod is a mere 2-dimensional simplicial complex with Euclidean pieces \cite{Monod:2016kq}.
\section{Examples of CAT(0) spaces and groups}
\subsection{Symmetric spaces}\label{syms}

\begin{definition} A Riemannian manifold $X$ is \emph{symmetric}\index{space!symmetric}\index{symmetric!space} if for any point $x$, there is a global isometry $\sigma_x$ fixing $x$ such that its differential at $x$ is $-\Id$. A \emph{symmetric space of non-compact type} is a symmetric Riemannian manifold with non-positive sectional curvature and no local Euclidean factor.
\end{definition} 

Since we will not consider symmetric spaces which are not of compact type, we will use \emph{symmetric spaces} as a shortcut for symmetric spaces of non-compact type. It is not a priori obvious that these spaces are simply connected but this is actually the case and thus these spaces are CAT(0). 

Standard references for symmetric spaces are \cite{MR1834454} and \cite{MR1441541} but for a more friendly introduction one may read \cite{MR2655308,MR2655309} with two different approaches: a differential one and an algebraic one. Symmetric spaces are strongly related to semi-simple Lie groups. Let us describe the dictionary between symmetric spaces of non-compact type and semi-simple Lie groups without compact factor.

Let $X$ be a symmetric space of non-compact type.
\begin{itemize}
\item The connected component of the identity in the isometry group $G=\Isom^\circ(X)$ is a semi-simple Lie group with trivial center and no compact factor.
\item The stabilizer $K$ of a point $x\in X$ is a maximal compact subgroup of $G$. Any two such subgroups are conjugate.
\item The action $G\action X$ is transitive, thus $X\simeq G/K$.
\end{itemize}

Conversely, If $G$ is a semi-simple Lie group with finite center and no compact factor then $G$ has a maximal compact subgroup $K$ and $X=G/K$ has the structure of a symmetric space of non-compact type. 

This dictionary emphasises that there is a classification of symmetric spaces parallel to the classification of semi-simple Lie groups (see \cite[Chapter IV]{MR1847105} for historical considerations on Cartan's work).

In particular, there is a unique way (up to permutation of the factors) to write any symmetric space $X$ as a product of irreducible factors:

$$X\simeq X_1\times\dots\times X_n.$$
In this splitting, which coincides with the de Rham decomposition, each $X_i$ is an irreducible symmetric space, that is, $X_i\simeq G_i/K_i$ where $G_i$ is a simple Lie group (with trivial center and no compact factor) and $K_i$ is a maximal compact subgroup. So $\Isom(X_1)\times\cdots\times\Isom(X_n)$ is a finite index subgroup of $\Isom(X)$. The finite index corresponds to the finite possibilities of permuting the possible isometric factors.

On a homogeneous space $G/K$ where $G$ is a simple Lie group and $K$ a maximal subgroup, there is a unique $G$-invariant Riemannian metric up to scaling. So two symmetric spaces with the same isometry group are isometric up to scaling the metric on the irreducible factors. 

%For a general symmetric space with non-positive sectional curvature $X$ one has a product decomposition $X\simeq Z\times Y$ where $Z$ is quotient of a Euclidean space and $Y$ is a symmetric space of non-compact type.

\begin{definition} A CAT(0) space is \emph{symmetric}\index{symmetric!CAT(0) space}\index{CAT(0) space!symmetric} if for each point $x\in X$, there is an involutive isometry $\sigma_x$ such that $x$ is the unique fixed point of $\sigma_x$.\end{definition}

With this definition, one can easily  see that symmetric CAT(0) spaces are geodesically complete (any geodesic segment is included in some geodesic line) and homogeneous.

The following theorem shows that the differentiable structure of a symmetric space of non-compact type is completely contained in its metric structure.
\begin{theorem}[{\cite[Th\'eor\`eme 8.4]{mathese}}]Let $X$ be a proper symmetric CAT(0) space, then $X$ has an isometric splitting
$$X\simeq E\times Y$$
where $E$ is a Euclidean space and $Y$ a symmetric space of non-compact type.
\end{theorem}

The proof of this theorem essentially relies on the following consequence of the solution of the Hilbert fifth problem: a connected locally compact group with trivial amenable radical is a semi-simple Lie group without compact factors (\cite[Theorem 11.3.4(ii)]{MR1840942}).

Since the manifold structure can be recovered from the metric data, it is natural to ask for a purely metric approach of symmetric spaces of non-compact type. In particular, how to prove the CAT(0) inequality without relying on differential geometry? See \cite[Exercise 1.4.(iii)]{MR3329726}. The hyperbolic space is the simplest example of a symmetric space. Let us prove that it is CAT(0) without using differential geometry.

\begin{proposition}\label{HCAT(0)}The hyperbolic space is CAT(0).
\end{proposition}
To prove this proposition, we use the following characterisation of CAT(0) spaces that relies on the \emph{Alexandrov angle}\index{angle!Alexandrov}\index{Alexandrov angle}, which is notion of angle for geodesic segments with a common endpoint. See \cite[Definition I.12]{MR1744486}.

\begin{lemma}[{\cite[Exercise II.1.9(c)]{MR1744486}}]Let $(X,d)$ be a geodesic metric space. Then $X$ is CAT(0) if and only if for any triple of points  $A,B,C\in X$ with distance $a=d(B,C),\ b=d(A,C),\ c=d(A,B)$ and Alexandrov angle $\gamma$ at the vertex $C$, one has
\begin{equation}\label{angle}c^2\geq a^2+b^2-2ab\cos(\gamma).\end{equation}
\end{lemma}
Observe that the equality case is the so-called (Euclidean) \emph{law of cosines}\index{law of cosines}.
\begin{proof}[Proof of Proposition \ref{HCAT(0)}]We use the hyperboloid model \cite[Chapter I.2]{MR1744486}. Let $\mathbb{E}^{n,1}$ be the vector space $\R^{n+1}$ with the non-degenerate quadratic form $\langle u|v \rangle=-u_{n+1}v_{n+1}+\sum_{i=1}^nu_iv_i.$ The hyperbolic space $\mathbb{H}^n$ is then $\{u=(u_1,\dots,u_{n+1})\in\mathbb{E}^{n,1}\ |\ \langle u,v\rangle=-1,\ u_{n+1}>0\}$. The distance between two points $A,B\in\mathbb{H}^n$, is defined by $\cosh(d(A,B))=-\langle A|B\rangle$. To show that this is a metric, the only non-trivial fact to check is the triangle inequality. This is a consequence of the \emph{hyperbolic law of cosines}\index{law of cosines!hyperbolic}\index{hyperbolic!law of cosines}. This law is a simple consequence of the model \cite[I.2.7]{MR1744486}. For any triple of points $A,B,C\in\mathbb{H}^n$ with distance $a=d(B,C),\ b=d(A,C),\ c=d(A,B)$,
$$\cosh(c)=\cosh(a)\cosh(b)-\sinh(a)\sinh(b)\cos(\gamma)$$
where $\gamma$ is the angle at vertex $C$ (which coincides with the Alexandrov angle thanks to \cite[Proposition I.2.9]{MR1744486}). We aim now to show Inequality \eqref{angle}.

%\begin{figure}
%\caption{Hyperbolic law of cosines.}
%\end{figure}¥
 We consider first the case $0\leq \gamma\leq \pi/2$. Fix side lengths $a,b,c$, define $f(\gamma)= \cosh(a)\cosh(b)-\sinh(a)\sinh(b)\cos(\gamma)$ and $g(\gamma)=\cosh\left(\sqrt{a^2+b^2-2ab\cos(\gamma)}\right)$. Thus, our goal is to show that $\varphi(\gamma)=f(\gamma)-g(\gamma)\geq0$. One has $\varphi(0)=0$. For convenience, set $c_0=\sqrt{a^2+b^2-2ab\cos(\gamma)}$ (that is, the length of the third side in a Euclidean triangle with sides $a,b$ and angle $\gamma$). We have:

\begin{align*}
f'(\gamma)&=\sinh(a)\sinh(b)\sin(\gamma)=\sum_{n,m\geq0}\frac{a^{2n+1}b^{2m+1}}{(2n+1)!(2m+1)!}\sin(\gamma)\\
g'(\gamma)&=\frac{\sinh(c_0)}{c_0}ab\sin(\gamma)=\sum_{k\geq0}\frac{\left(a^2+b^2-2ab\cos(\gamma)\right)^k}{(2k+1)!}ab\sin(\gamma)\\
&\leq\sum_{k\geq0}\frac{(a^2+b^2)^k}{(2k+1)!}ab\sin(\gamma)\\
&\leq\sum_{k\geq0}\sum_{0\leq i\leq k}\frac{k!}{i!(k-i)!(2k+1)!}a^{2i+1}b^{2(k-i)+1}\sin(\gamma)
\end{align*}
In the last double sum, the coefficient in front of $a^{2n+1}b^{2m+1}$ is $\frac{(n+m)!}{n!m!(2(n+m)+1)!}$. Since $(n+m)!\geq n!m!$ and $(2(n+m)+1)!\geq(2n+1)!(2m+1)!$, one has $\frac{1}{(2n+1)!(2m+1)!}\geq\frac{(n+m)!}{n!m!(2(n+m)+1)!}$. Thus $\varphi'\geq 0$ on $[0,\pi/2]$ and $\varphi(\gamma)\geq0$ on $[0,\pi/2]$.

Now, let us consider the case $\gamma\geq \pi/2$. Choose $D\in [A,B]$ such the angles $\gamma'=\angle([C,A],[C,D])$ and $\gamma''=\angle([CD,CB])$ are at most $\pi/2$. Let $c'=d(A,D)$, $c''=d(D,B)$ and $c'_0,c''_0$ the corresponding Euclidean lengths (that is $c_0''=\sqrt{a^2+b^2-2ab\cos(\gamma'')})$. The former case shows that $c'\geq c'_0$ and $c''\geq c''_0$. Moreover, in the Euclidean plane, one has $c_0'+c_0''\geq c_0$ and finally $c=c'+c''\geq c_0'+c_0''\geq c_0$.
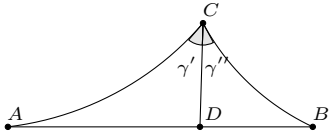
\begin{figure}[h]
\begin{center}
\begin{tikzpicture}[line cap=round,line join=round,>=triangle 45,x=1.0cm,y=1.0cm]
\clip(-0.42695519773771173,-0.3645455038614315) rectangle (4.51726170877022,1.6627292749132139);
\draw [shift={(2.56,1.38)},fill=black,fill opacity=0.1] (0,0) -- (-130.24594596471982:0.2916942127733293) arc (-130.24594596471982:-91.6602823689828:0.2916942127733293) -- cycle;
\draw [shift={(2.56,1.38)},fill=black,fill opacity=0.1] (0,0) -- (-91.66028236898282:0.2916942127733293) arc (-91.66028236898282:-62.13861201329126:0.2916942127733293) -- cycle;
\draw [shift={(-0.4782681484984628,3.9517148262000474)}] plot[domain=4.83283119139643:5.5807610395155,variable=\t]({1.*3.9805515308154744*cos(\t r)+0.*3.9805515308154744*sin(\t r)},{0.*3.9805515308154744*cos(\t r)+1.*3.9805515308154744*sin(\t r)});
\draw [shift={(5.359366515837104,2.8597737556561085)}] plot[domain=3.627865608133455:4.26866279788915,variable=\t]({1.*3.166414890363951*cos(\t r)+0.*3.166414890363951*sin(\t r)},{0.*3.166414890363951*cos(\t r)+1.*3.166414890363951*sin(\t r)});
\draw (0.,0.)-- (4.,0.);
\draw (2.56,1.38)-- (2.52,0.);
\begin{scriptsize}
\draw [fill=black] (0.,0.) circle (1.0pt);
\draw[color=black] (0.09809438525428099,0.17508878976922948) node {$A$};
\draw [fill=black] (4.,0.) circle (1.0pt);
\draw[color=black] (4.1088898108875584,0.17508878976922948) node {$B$};
\draw [fill=black] (2.56,1.38) circle (1.0pt);
\draw[color=black] (2.6650034576595782,1.560636300442548) node {$C$};
\draw [fill=black] (2.52,0.) circle (1.0pt);
\draw[color=black] (2.7,0.17508878976922948) node {$D$};
\draw[color=black] (2.35,0.8) node {$\gamma'$};
\draw[color=black] (2.75,0.8) node {$\gamma''$};
\end{scriptsize}
\end{tikzpicture}
\caption{Point $D$ and  angles $\gamma',\gamma''$}
\end{center}
\end{figure}
\end{proof}
Proper symmetric CAT(0) spaces are well understood and one can ask what are the symmetric CAT(0) spaces outside the proper world. One can consruct infinite-dimensional analogs of the symmetric spaces of non-compact type that remain in the CAT(0) world. These are manifolds modelled on Hilbert spaces with a Riemannian metric. They can be classified in a  way analogous to the finite-dimensional case \cite[Theorem 1.8]{BD15}. There are also symmetric CAT(0) spaces that are not Riemannian manifolds  \cite[\S 4.3]{BD15}.
\subsection{Euclidean Buildings}

Buildings\index{building} were introduced by Jacques Tits to get a geometry to understand properties of algebraic groups. In some sense, this is a converse of Klein's Erlangen Program. Such a geometry was already existing for semi-simple Lie groups. It is given by symmetric spaces. Euclidean buildings are analogs in the non-Archimedean world. 

Here we use the Kleiner-Leeb axiomatisation of Euclidean buildings \cite{MR1608566}. This axiomatisation starts from a CAT(0) space and thus is more geometric than the original definition due to Tits. It is a priori not completely clear that one can recover Tits' definition from the one of Kleiner and Leeb. Parreau proved that this is actually the case   \cite[\S 2.7]{MR1796138}. Her definition is a bit more general to cover the case of non complete spaces.

Let $E$ be a Euclidean space. Its boundary at infinity $\bo E$ endowed with the angular metric is a Euclidean sphere of dimension  dim$(E)-1$. Since isometries of $E$ are affine and translations act trivially on $\bo E$, one obtain a homomorphism 
\[\rho\colon\Isom(E)\to\Isom(\bo E)\]
that associates its linear part to every Euclidean isometry. A subgroup $W_\mathrm{Aff}\leq\Isom(E)$ is called an \emph{affine Weyl group}\index{Weyl group!affine}\index{affine Weyl group} if it is generated by reflections through hyperplanes and if $W:=\rho(W_\mathrm{Aff})$ is a finite subgroup of $\Isom(\bo E)$. 

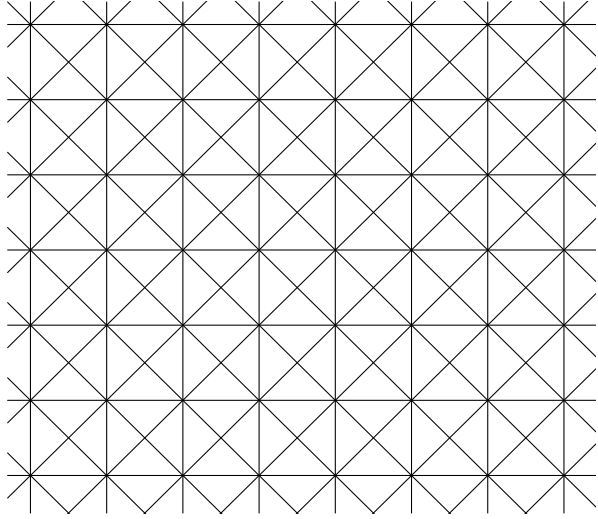
\begin{figure}[h]
\begin{center}
\begin{tikzpicture}[line cap=round,line join=round,>=triangle 45,x=1.0cm,y=1.0cm]
\clip(-4.3,-0.5) rectangle (3.42,6.3);
\draw [domain=-4.3:3.42] plot(\x,{(-10.-1.*\x)/1.});
\draw [domain=-4.3:3.42] plot(\x,{(-9.-1.*\x)/1.});
\draw [domain=-4.3:3.42] plot(\x,{(-8.-1.*\x)/1.});
\draw [domain=-4.3:3.42] plot(\x,{(-7.-1.*\x)/1.});
\draw [domain=-4.3:3.42] plot(\x,{(-6.-1.*\x)/1.});
\draw [domain=-4.3:3.42] plot(\x,{(-5.-1.*\x)/1.});
\draw [domain=-4.3:3.42] plot(\x,{(-4.-1.*\x)/1.});
\draw [domain=-4.3:3.42] plot(\x,{(-3.-1.*\x)/1.});
\draw [domain=-4.3:3.42] plot(\x,{(-2.-1.*\x)/1.});
\draw [domain=-4.3:3.42] plot(\x,{(-1.-1.*\x)/1.});
\draw [domain=-4.3:3.42] plot(\x,{(-0.-1.*\x)/1.});
\draw [domain=-4.3:3.42] plot(\x,{(--1.-1.*\x)/1.});
\draw [domain=-4.3:3.42] plot(\x,{(--2.-1.*\x)/1.});
\draw [domain=-4.3:3.42] plot(\x,{(--3.-1.*\x)/1.});
\draw [domain=-4.3:3.42] plot(\x,{(--4.-1.*\x)/1.});
\draw [domain=-4.3:3.42] plot(\x,{(--5.-1.*\x)/1.});
\draw [domain=-4.3:3.42] plot(\x,{(--6.-1.*\x)/1.});
\draw [domain=-4.3:3.42] plot(\x,{(--7.-1.*\x)/1.});
\draw [domain=-4.3:3.42] plot(\x,{(--8.-1.*\x)/1.});
\draw [domain=-4.3:3.42] plot(\x,{(--9.-1.*\x)/1.});
\draw [domain=-4.3:3.42] plot(\x,{(--10.-1.*\x)/1.});
\draw [domain=-4.3:3.42] plot(\x,{(-10.-1.*\x)/-1.});
\draw [domain=-4.3:3.42] plot(\x,{(-9.-1.*\x)/-1.});
\draw [domain=-4.3:3.42] plot(\x,{(-8.-1.*\x)/-1.});
\draw [domain=-4.3:3.42] plot(\x,{(-7.-1.*\x)/-1.});
\draw [domain=-4.3:3.42] plot(\x,{(-6.-1.*\x)/-1.});
\draw [domain=-4.3:3.42] plot(\x,{(-5.-1.*\x)/-1.});
\draw [domain=-4.3:3.42] plot(\x,{(-4.-1.*\x)/-1.});
\draw [domain=-4.3:3.42] plot(\x,{(-3.-1.*\x)/-1.});
\draw [domain=-4.3:3.42] plot(\x,{(-2.-1.*\x)/-1.});
\draw [domain=-4.3:3.42] plot(\x,{(-1.-1.*\x)/-1.});
\draw [domain=-4.3:3.42] plot(\x,{(-0.-1.*\x)/-1.});
\draw [domain=-4.3:3.42] plot(\x,{(--1.-1.*\x)/-1.});
\draw [domain=-4.3:3.42] plot(\x,{(--2.-1.*\x)/-1.});
\draw [domain=-4.3:3.42] plot(\x,{(--3.-1.*\x)/-1.});
\draw [domain=-4.3:3.42] plot(\x,{(--4.-1.*\x)/-1.});
\draw [domain=-4.3:3.42] plot(\x,{(--5.-1.*\x)/-1.});
\draw [domain=-4.3:3.42] plot(\x,{(--6.-1.*\x)/-1.});
\draw [domain=-4.3:3.42] plot(\x,{(--7.-1.*\x)/-1.});
\draw [domain=-4.3:3.42] plot(\x,{(--8.-1.*\x)/-1.});
\draw [domain=-4.3:3.42] plot(\x,{(--9.-1.*\x)/-1.});
\draw [domain=-4.3:3.42] plot(\x,{(--10.-1.*\x)/-1.});
\draw (-10.,-0.5) -- (-10.,6.3);
\draw (-9.,-0.5) -- (-9.,6.3);
\draw (-8.,-0.5) -- (-8.,6.3);
\draw (-7.,-0.5) -- (-7.,6.3);
\draw (-6.,-0.5) -- (-6.,6.3);
\draw (-5.,-0.5) -- (-5.,6.3);
\draw (-4.,-0.5) -- (-4.,6.3);
\draw (-3.,-0.5) -- (-3.,6.3);
\draw (-2.,-0.5) -- (-2.,6.3);
\draw (-1.,-0.5) -- (-1.,6.3);
\draw (0.,-0.5) -- (0.,6.3);
\draw (1.,-0.5) -- (1.,6.3);
\draw (2.,-0.5) -- (2.,6.3);
\draw (3.,-0.5) -- (3.,6.3);
\draw (4.,-0.5) -- (4.,6.3);
\draw (5.,-0.5) -- (5.,6.3);
\draw (6.,-0.5) -- (6.,6.3);
\draw (7.,-0.5) -- (7.,6.3);
\draw (8.,-0.5) -- (8.,6.3);
\draw (9.,-0.5) -- (9.,6.3);
\draw (10.,-0.5) -- (10.,6.3);
\draw [domain=-4.3:3.42] plot(\x,{(-10.-0.*\x)/1.});
\draw [domain=-4.3:3.42] plot(\x,{(-9.-0.*\x)/1.});
\draw [domain=-4.3:3.42] plot(\x,{(-8.-0.*\x)/1.});
\draw [domain=-4.3:3.42] plot(\x,{(-7.-0.*\x)/1.});
\draw [domain=-4.3:3.42] plot(\x,{(-6.-0.*\x)/1.});
\draw [domain=-4.3:3.42] plot(\x,{(-5.-0.*\x)/1.});
\draw [domain=-4.3:3.42] plot(\x,{(-4.-0.*\x)/1.});
\draw [domain=-4.3:3.42] plot(\x,{(-3.-0.*\x)/1.});
\draw [domain=-4.3:3.42] plot(\x,{(-2.-0.*\x)/1.});
\draw [domain=-4.3:3.42] plot(\x,{(-1.-0.*\x)/1.});
\draw [domain=-4.3:3.42] plot(\x,{(-0.-0.*\x)/1.});
\draw [domain=-4.3:3.42] plot(\x,{(--1.-0.*\x)/1.});
\draw [domain=-4.3:3.42] plot(\x,{(--2.-0.*\x)/1.});
\draw [domain=-4.3:3.42] plot(\x,{(--3.-0.*\x)/1.});
\draw [domain=-4.3:3.42] plot(\x,{(--4.-0.*\x)/1.});
\draw [domain=-4.3:3.42] plot(\x,{(--5.-0.*\x)/1.});
\draw [domain=-4.3:3.42] plot(\x,{(--6.-0.*\x)/1.});
\draw [domain=-4.3:3.42] plot(\x,{(--7.-0.*\x)/1.});
\draw [domain=-4.3:3.42] plot(\x,{(--8.-0.*\x)/1.});
\draw [domain=-4.3:3.42] plot(\x,{(--9.-0.*\x)/1.});
\draw [domain=-4.3:3.42] plot(\x,{(--10.-0.*\x)/1.});
\end{tikzpicture}
\caption{Reflection hyperplanes for the discrete affine Weyl group $B_2\ltimes \mathbb{Z}^2$.}
\end{center}
\end{figure}The group $W$ is called the \emph{spherical Weyl group}\index{Weyl group!spherical}\index{spherical!Weyl group} associated with $W_\mathrm{Aff}$. If $W_\mathrm{Aff}$ is an affine Weyl group then $(E, W_\mathrm{Aff})$ is called a \emph{Euclidean Coxeter complex} and $(\bo E,W)$ is the associated \emph{spherical Coxeter complex at infinity}\index{Coxeter complex!spherical}\index{spherical!Coxeter complex}. Its \emph{anisotropy polyhedron}\index{anisotropy polyhedron}\index{polyhedron!anisotropy} is the spherical polyhedron 
\[\Delta:=\bo E/W.\]

\begin{figure}[h]
\begin{center}
\includegraphics[width=.4\textwidth]{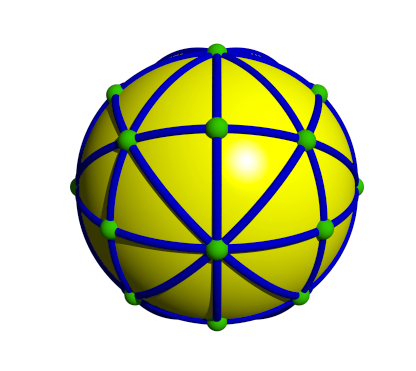}
\caption{The sphere at infinity $\bo E$ with the trace of reflection hyperplanes for the spherical Weyl group $B_3$.}
\end{center}
\end{figure}
An oriented segment (not reduced to a point) $\overline{xy}$ of $E$ determines  a unique point of $\bo E$ and the projection of this point to $\Delta$ is called the $\Delta$-\emph{direction} of  $\overline{xy}$. Let  $\pi$ be the projection $\bo E\to\Delta$. If $\delta_1,\delta_2$ are two points of $\Delta$, we introduce the finite set
\[D(\delta_1,\delta_2)=\{\angle(\xi_1,\xi_2))|\xi_1,\xi_2\in\bo E,\ \pi(\xi_1)=\delta_1,\ \pi(\xi_2)=\delta_2\}.\]
\begin{definition}\label{dfb}Let $(E, W_\mathrm{Aff})$ be a Euclidean Coxeter complex. A \emph{Euclidean building modelled on}\index{Euclidean building}\index{building!Euclidean} $(E, W_\mathrm{Aff})$ is a complete CAT(0) space $(X,d)$ with 
\begin{itemize}
\item[(i)] a map $\theta$ from the set of oriented segments not reduced to a point to $\Delta$;
\item[(ii)] a collection, $\mathcal{A}$, called \emph{atlas}\index{atlas}, of isometric embeddings $\iota\colon E\to X$ that preserve $\Delta$-directions. This atlas is closed under precompositions with isometries in $W_\mathrm{Aff}$. The image of such an isometric embedding $\iota$ is called an \emph{apartment}\index{apartment}.
\end{itemize}
Moreover the following properties must hold.
\begin{itemize}
\item[(1)] For all $x,y,z\in X$ such that $y\neq z$ and $x\neq z$,
\[d_\Delta(\theta(\overline{xy}),\theta(\overline{xz}))\leq\overline{\angle}_x(y,z).\]
\item[(2)] The angle between two geodesic segments  $\overline{xy}$ and $\overline{xz}$ is in $D(\theta(\overline{xy}),\theta(\overline{xz}))$.
\item[(3)] Every geodesic segment, ray or line is contained in an apartment.
\item[(4)] If $A_1$ and $A_2$ are two apartments with a non-empty intersection then the \emph{transition map} $\iota^{-1}_{A_2}\circ\iota_{A_1}\colon \iota_{A_1}^{-1}(A_1\cap A_2)\to\iota_{A_2}^{-1}(A_1\cap A_2)$ is the restriction of an element of $W_{\mathrm{Aff}}$.
\end{itemize}
\end{definition}

Euclidean subspaces of a Euclidean building are included in some apartment and thus the rank of a Euclidean building is the dimension of any apartment that is the dimension of the model Euclidean space $E$. When the rank is 1, that is, $E$ is the real line, one recover (in a cumbersome way) real trees with no leaves.

\begin{remark} A Euclidean building is called \emph{discrete}\index{discrete building}\index{building!discrete} if the translation part of $W_\mathrm{Aff}$ is discrete.  If $X$ is an irreducible discrete Euclidean building such that $W_\mathrm{Aff}$ acts cocompactly  then $X$ has a simplicial structure. Definition \ref{dfb} allows \emph{non-discrete} buildings in a similar way that real trees are generalisations of simplicial trees.
\end{remark} 

The purpose of the Kleiner-Leeb definition is to be stable under asymptotic cones. Recall that the asymptotic cone of a metric space $(X,d)$ is a limit of  a sequence of metric spaces $(X, \lambda_n^{-1} d)$ with $\lambda_n\to+\infty$. For example, ideal triangles (that is, triangles with vertices at infinity) in the hyperbolic plane become thinner and thinner when $\lambda_n\to+\infty$ and the limit is actually a real tree. 
Asymptotic cones of Euclidean buildings are Euclidean buildings of the same rank and spherical Weyl group \cite[Theorem 5.1.1]{MR1608566}. Moreover asymptotic cones of symmetric spaces $X$ are Euclidean buildings of the same rank \cite[Theorem 5.2.1]{MR1608566}. In that case, the spherical Weyl group is the classical Weyl group of the semi-simple group $\Isom(X)$. This result holds also in infinite dimensions: The asymptotic cone of a symmetric space of infinite dimension and finite rank is a Euclidean building of the same rank \cite[Theorem 1.3]{MR3044451}.

There is a correspondance between Euclidean buildings and algebraic groups over fields with a non-Archimedean valuation (when the rank is at least 3). Euclidean buildings associated to algebraic groups are called \emph{Bruhat-Tits buildings}\index{Bruhat-Tits building}\index{building!Bruhat-Tits}. To prove the correspondance, Tits showed how to reconstruct the group from the building \cite{MR0470099,MR843391}. See \cite[\S 11.9]{MR2439729} and references therein for an overview.

Euclidean buildings and symmetric spaces share a common feature at infinity: their boundaries endowed with the Tits metric are \emph{spherical buildings}\index{spherical!building}\index{building!spherical}. We do not aim to give a definition but a taste of their geometric and combinatorial natures. Spherical buildings can be defined as metric objects with a condition of curvature bounded above \cite[\S 3.2]{MR1608566} and some more combinatorial properties. They are union of Euclidean spheres of the same dimension. These spheres are glued together according to the action of a finite Coxeter group. For a soft introduction, we recommend \cite[II.10.71]{MR1744486}. Spherical buildings encode algebraic data. For example, let us consider the spherical building $B$  associated to $\SL(V)$ where $V$ is some finite-dimensional vector field. This building $B$ has the structure of a simplicial complex where simplices are in bijection with flags in $V$. Inclusion of flags corresponds to inclusion of simplices.

The presence of this spherical building at infinity is the essential source of rigidity phenomena detailed in Section \ref{rigidity}.

\subsection{CAT(0) cube complexes}
Besides symmetric spaces and Euclidean buildings, CAT(0) cube complexes provide a very nice class of CAT(0) spaces with interesting actions. They belong to the class of polyhedral cell complexes which are CAT(0) like discrete Euclidean buildings. For references about groups acting on CAT(0) complexes, one may consult \cite[Chapter II.12]{MR1744486}. 

CAT(0) cube complexes are sufficiently well understood to obtain very strong results like the Tits Alternative (Theorem \ref{Titscc}) or a proof of the Rank Rigidity conjecture (Section \ref{rrc}). CAT(0) cube complexes have both a geometric structure (CAT(0) inequality) and a more combinatorial one: there are spaces with walls. The latter structure allows to understand properties of groups with an analytic flavour like the Haagerup property and Property (T). Moreover, the combinatorics of spaces with walls allow to construct CAT(0) cube complexes. This is the so-called cubulation process and we will see some remarkable applications of that construction. For an introductory paper to the subject of CAT(0) cube complexes, we recommend \cite{MR3329724}.

Roughly speaking, a CAT(0) cube complex is a collection of cubes glued together isometrically along faces in such a way that the resulting length space is CAT(0). 

A \emph{cube}\index{cube} is a metric space isometric to $[0,1]^n$ for some $n$ (this $n$ is called the \emph{dimension} of the cube). A \emph{face} of a cube $C$ is a subset corresponding to fixing $k$ coordinates by making them equal to 0 or 1. In particular, a face of $C$ is itself a cube of dimension $n-k$ where $n$ is the dimension of $C$ and $k$ is the number of fixed coordinates. An \emph{attaching map} between two cubes is an isometry between two faces of two different cubes. A \emph{cubical complex}\index{cubical complex}\index{complex!cubical} is the space obtained from a collection of cubes and a collection of attaching maps where a point and its image via an attaching map are identified. See \cite[Definition 7.32]{MR1744486} for details.

The \emph{dimension}\index{dimension!cubical complex} of a cubical complex $X$ is the supremum of the dimensions of its cubes. It may be infinite, otherwise one say that $X$ has finite dimension.  

A cubical complex $X$ is naturally endowed with a length metric obtained by minimising lengths of continuous paths between two points where the distance between two points in a common cube is the Euclidean one. A \emph{vertex}\index{vertex} (or a cube of dimension 0) is a vertex of some cube. The \emph{link}\index{link} $\Lk(v)$ at a vertex $v\in X$ is the simplicial complex whose vertices are edges with one end $v$. Edges $e_0,e_1,\dots,e_n\in\Lk(v)$ are the vertices of an $n$-simplex in $\Lk(v)$ if there is a cube $C$ in $X$ such that $v$ is a vertex of $C$ and the $e_i$'s are edges of $C$.

A simplicial complex is a \emph{flag complex}\index{flag complex}\index{complex!flag} if for any finite collection of vertices $v_0,\dots,v_n$ such that for any $i,j$, $v_i$ and $v_j$ are linked by an edge there is an $n$-simplex with vertices $v_0,\dots,v_n$.

Gromov gave a local characterisation of cubical complexes that are CAT(0). Such cubical complexes are called \emph{CAT(0) cube complexes}\index{CAT(0)!cube complex}. A proof in finite dimensions can be found in \cite[Theorem II.5.20]{MR1744486}. The general case is treated in \cite[Theorem B.8]{MR3029427}.

\begin{theorem}[Gromov link condition] A cubical complex $X$ is CAT(0) if and only if $X$ is simply connected and for any vertex the associated link is a flag complex.
\end{theorem}

In dimension 1, CAT(0) cube complexes are exactly simplicial trees with the metric where any edge has length 1.

\begin{definition}Let $C$ be cube. A \emph{midcube}\index{midcube} is the subset of $C$ where one coordinate has been fixed to $1/2$.
\end{definition}

It is a remarkable theorem of Sageev that midcubes of CAT(0) cube complexes can be extended along the whole complex. Let us describe that construction. We define an equivalence relation on the set of midcubes. For two midcubes $M_1,M_2$ in cubes $C_1,C_2$, we define $M_1\simeq M_2$ if $C_1$ and $C_2$ have a common edge $e$ such that $M_i\cap e$ is the mid-point of $e$ for $i=1,2$.  The equivalence relation we consider is the one generated by this relation.

\begin{definition} A \emph{hyperplane}\index{hyperplane} $H$ of a CAT(0) cube complex is the union of midcubes in a equivalence class. The \emph{carrier}\index{carrier!hyperplane} $C(H)$ of $H$ is the union of cubes intersecting $H$.
\end{definition}

\begin{theorem}[{\cite{MR1347406}}]Let $H$ be a hyperplane of some CAT(0) cube complex $X$.
\begin{itemize}
\item $H$ is a closed convex subspace of $X$,
\item $H$ has the structure of a CAT(0) cube complex of dimension at most $\dim(X)-1$,
\item $C(H)$ is isometric to $H\times [-1/2,1/2]$,
\item $X\setminus H$ has two connected components which are convex subsets of $X$.
\end{itemize}
\end{theorem}

\begin{figure}[h]
\begin{center}
\definecolor{zzttqq}{rgb}{0.6,0.2,0.}
\begin{tikzpicture}[line cap=round,line join=round,>=triangle 45,x=1.0cm,y=1.0cm]
\clip(-1.595921374942505,-1.7779261296597524) rectangle (10.9446964130976,3.2562361538248785);
\fill[color=zzttqq,fill=zzttqq,fill opacity=0.1] (-1.,1.5) -- (3.,1.5) -- (4.,1.) -- (0.,1.) -- cycle;
\fill[color=zzttqq,fill=zzttqq,fill opacity=0.1] (7.4,0.6) -- (9.4,0.6) -- (10.4,0.2) -- (8.4,0.2) -- cycle;
\draw (0.,2.)-- (2.,2.);
\draw (2.,2.)-- (4.,2.);
\draw (4.,2.)-- (4.,0.);
\draw (2.,2.)-- (2.,0.);
\draw (4.,0.)-- (2.,0.);
\draw (0.,2.)-- (0.,0.);
\draw (0.,0.)-- (2.,0.);
\draw (-1.,2.5)-- (0.,2.);
\draw (1.,2.5)-- (2.,2.);
\draw (3.,2.5)-- (4.,2.);
\draw (1.,2.5)-- (3.,2.5);
\draw (-1.,2.5)-- (1.,2.5);
\draw (-1.,2.5)-- (-1.,0.5);
\draw (-1.,0.5)-- (0.,0.);
\draw [dash pattern=on 4pt off 4pt] (-1.,0.5)-- (1.,0.5);
\draw [dash pattern=on 4pt off 4pt] (1.,0.5)-- (1.,2.5);
\draw (1.,0.5)-- (2.,0.);
\draw [dash pattern=on 4pt off 4pt] (3.,2.5)-- (3.,0.5);
\draw [dash pattern=on 4pt off 4pt] (3.,0.5)-- (1.,0.5);
\draw [dash pattern=on 4pt off 4pt] (3.,0.5)-- (4.,0.);
\draw (4.,2.)-- (5.6,1.8);
\draw (5.6,1.8)-- (5.6,-0.2);
\draw (5.6,-0.2)-- (4.,0.);
\draw (7.4,1.6)-- (5.6,1.8);
\draw (7.4,1.6)-- (7.4,-0.4);
\draw (7.4,-0.4)-- (5.6,-0.2);
\draw (7.4,1.6)-- (9.4,1.6);
\draw [dash pattern=on 4pt off 4pt] (9.4,1.6)-- (9.4,-0.4);
\draw [dash pattern=on 4pt off 4pt] (7.4,-0.4)-- (9.4,-0.4);
\draw (8.4,1.2)-- (7.4,1.6);
\draw (9.4,1.6)-- (10.4,1.2);
\draw (10.4,1.2)-- (8.4,1.2);
\draw (8.4,1.2)-- (8.4,-0.8);
\draw (8.4,-0.8)-- (7.4,-0.4);
\draw (10.4,1.2)-- (10.4,-0.8);
\draw (10.4,-0.8)-- (8.4,-0.8);
\draw [dash pattern=on 4pt off 4pt] (9.4,-0.4)-- (10.4,-0.8);
\draw (-1.,1.5)-- (0.,1.);
\draw (0.,1.)-- (4.,1.);
\draw (-1.,1.5)-- (3.,1.5);
\draw (3.,1.5)-- (4.,1.);
\draw [color=zzttqq] (-1.,1.5)-- (3.,1.5);
\draw [color=zzttqq] (3.,1.5)-- (4.,1.);
\draw [color=zzttqq] (4.,1.)-- (0.,1.);
\draw [color=zzttqq] (0.,1.)-- (-1.,1.5);
\draw [color=zzttqq] (7.4,0.6)-- (9.4,0.6);
\draw [color=zzttqq] (9.4,0.6)-- (10.4,0.2);
\draw [color=zzttqq] (10.4,0.2)-- (8.4,0.2);
\draw [color=zzttqq] (8.4,0.2)-- (7.4,0.6);
\draw [color=zzttqq] (4.,1.)-- (7.4,0.6);
\begin{scriptsize}
\draw[color=zzttqq] (8.938197567011184,0.46146990391882) node {H};
\end{scriptsize}
\end{tikzpicture}
\caption{A hyperplane $H$ in some finite CAT(0) cube complex.}
\end{center}
\end{figure}
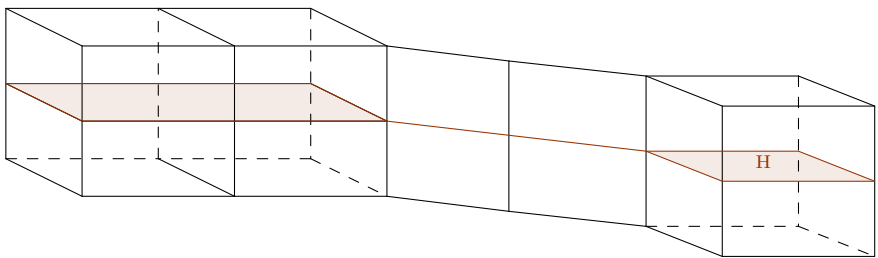

The two connected components of $X\setminus H$ are called \emph{half-spaces}\index{half-space}. Two points of $X$ are separated by $H$ if they lie in different components of $X\setminus H$. The 1-skeleton of $X$ is a connected graph and thus has a graph metric where all edges have length 1. Actually the distance between two vertices coincides with the number of hyperplanes that separate them. This metric can be extended to the whole complex in a unique way such that each cube of dimension $n$ becomes isometric to $[0,1]^n$ with the $\ell^1$-metric. This metric on $X$ is called the \emph{$\ell^1$-metric}. If $X$ has finite dimension then the CAT(0)-metric and the $\ell^1$-metric are quasi-isometric (see \cite[Section 1]{MR2490224}).

The non-positive curvature of CAT(0) cube complexes is very important but the combinatorics of cubes is also very important and sometimes the CAT(0)-metric forgets about the combinatorics. Let us illustrate that with a simple example. The Euclidean space $\R^n$ has a natural cube complex structure where the vertices are points with integral coordinates. The isometry group does not at all preserve this structure but if one endow $\R^n$ with the $\ell^1$-metric, the space becomes much more rigid and the linear part of the isometry group of the $\ell^1$-metric preserves the cubical structure.

Let us give another example. Salajan exhibited some elements of the Thompson group $F$ acting via parabolic isometries (for the CAT(0) metric) \cite[Proposition 2.4.1]{Salajan:2012yu} on the CAT(0) cube complex constructed by Farley. On the other hand, Haglund proved that the isometries of CAT(0) cube complexes are semi-simple in a combinatorial sense (up to considering the barycentric subdivision) \cite[Theorem 1.4]{Haglund:2007fk}. Observe that the existence of parabolic elements may appear only in infinite dimensions. In finite dimensions, there are only finitely many types of simplices and in that case, all isometries are semi-simple \cite{MR1744486}.

\begin{definition}Let $V$ be a set. A \emph{wall}\index{wall} $W$ is a partition of $V$ its two non-empty subsets $W_-,W_+$. Two points $u,v\in V$ are \emph{separated} by a wall $W$ if $u,v$ does not belong to the same subset $W_-$ or $W_+$. The two subsets $W_+$ and $W_-$ are called half-spaces associated to the wall $W$.   A \emph{space with walls}\index{wall!space with}\index{space!with walls} is a set $V$ with a collection of \emph{walls} where any two points are separated by a finite positive number of walls. 
\end{definition}

If $V$ is a space with walls then one has a metric associated to the collection of walls. The distance between $u,v\in V$ is the number of walls separating them. There is a duality between CAT(0) cube complexes and spaces with walls. Let $X$ be a CAT(0) cube complex and let $V$ be the set of vertices of $X$. Any hyperplane $H$ induces a wall on $V$ since any vertex is exactly in one component of $X\setminus H$. Thus the collection of hyperplanes gives the structure of a space with walls on $V$. Observe that a vertex $v$ of $X$ gives a choice of a half-space for any hyperplane: simply choose the one containing $v$ !

Conversely, if $V$ is a space with walls, an \emph{ultrafilter}\index{ultrafilter} on the set of walls $\mathcal{W}$ is a consistent choice of a half-space for any wall. More precisely, an ultrafilter $\alpha$ is a map from $\mathcal{W}$ to the set of half-spaces such that for any $W\in\mathcal{W}$, $\alpha(W)=W_\pm$ and if $W,W'\in\mathcal{W}$ and $W_+\subset W_+'$ then $\alpha(W)=W_+\implies \alpha(W')=W_+'$. An ultrafilter $\alpha$ satisfies the \emph{Descending Chain Condition}\index{descending chain condition} if any non-increasing (with respect to inclusion) chain of half-spaces in the image of $\alpha$ is eventually constant. As before, any point $v$ yields an ultrafilter  by choosing the hyperplane containing $v$. This ultrafilter satisfies the descending chain condition.

Now, if $V$ is a space with walls then one can define a CAT(0) cube complex whose vertex set is the set of ultrafilters that satisfies the descending chain condition. Two vertices are joined by an edge if the corresponding ultrafilters give a different choice for exactly one wall. We just constructed the 0 and 1-skeletons. The higher dimensional cubes are constructed inductively: If one see a complex isomorphic to the $n-1$ skeleton of an $n$-cube, one simply fills it with an $n$-cube. For more details, see \cite{MR2197811,Roller} or the original construction \cite{MR1347406}.

This duality between CAT(0) cube complexes and spaces with walls is actually equivariant: If a group acts on one of the structures it acts also on the other. See for example \cite[Theorem 3]{MR2197811}.

The process of constructing a CAT(0) cube complex from a space with walls (or more generally from a pocset \cite{Roller}) is called \emph{cubulation}\index{cubulation}.  The construction of a CAT(0) cube complex on which a group acts nicely (e.g. properly, cocompactly or both) is also called a cubulation of the group. For a nice survey about cubulated groups and their properties, we recommend \cite{MR3187627}. \begin{examples}Let us list  a few groups that have been cubulated:
\begin{itemize}
\item Right angle Artin groups \cite{MR2322545},
\item Coxeter groups \cite{MR1983376},
\item Small cancellation groups \cite{MR2053602}.
\end{itemize}
\end{examples}

We will see in Section \ref{rigidity} some very strong rigidity results for symmetric spaces and Euclidean buildings. It is natural to ask wether similar phenomena happen in the world of CAT(0) cube complexes. 

The Roller boundary\index{Roller boundary}\index{boundary!Roller} of a CAT(0) cube complex $X$  is defined in the following way. Let $\mathcal{H}$ be the set of hyperplanes of $X$. Any hyperplane defined two half-spaces $H_+$ and $H_-$. Any vertex of $X$ can be identified with a point in $\prod_{H\in \mathcal{H}} \{H_+,H_-\}$. Let us denote by $\overline X^R$ the closure of the image of the vertex set $X^0$ in this product. The \emph{Roller boundary}  $\bo_R X$ is $\overline X^R\setminus X^0$. 
\begin{question} Are there suitable hypotheses for rigidity in the world of CAT(0) cube complexes?  First there is no similar statement to Mostow rigidity (see Section \ref{super}). Actually one can cubulate surface groups in different ways. Thus one may construct quasi-isometric cube complexes with a proper and cocompact action of the same group such that the complexes have different Roller boundaries.
Nonetheless, one can ask how (under suitable hypotheses), the Roller boundary determines the complex.
\end{question}

\begin{question} Trees with edges of length 1 have a continuous analog: real trees. Since trees are CAT(0) cube complexes of dimension 1, it is natural to ask if there are analogs of real trees. Is there a good notion of \emph{real} CAT(0) cube complexes? For example, one can construct a CAT(0) cube complex from a space with walls. There is a more general notion of space with measured walls (see references in Section \ref{Haagerup}). Can one construct a CAT(0) space associated to a space with measured walls in an equivariant way? The natural generalisations of trees (working well with spaces with measured walls) are \emph{median spaces}\index{median space}. They are not CAT(0) in general. 
\end{question}
 
\subsection{Cubulating hyperbolic 3-manifold groups}
Let us describe a bit a recent and very fruitful cubulation of certain groups. A \emph{ hyperbolic 3-dimensional group} is the fundamental group of a closed 3-dimensional manifold which admits a hyperbolic metric, that is, a Riemannian metric of sectional curvature -1. In particular, such a group acts properly cocompactly on the hyperbolic space of dimension 3. So the group is quasi-isometric to this hyperbolic space and thus hyperbolic in the sense of Gromov (see e.g. \cite[III.2]{MR1744486} or \cite{MR1086648}).

A combination of results of Bergeron-Wise \cite{MR2931226} and Kahn-Markovich \cite{MR2912704} implies a cubulation of hyperbolic 3-dimensional groups.
 
\begin{theorem}\label{KM}Every hyperbolic 3-manifold group acts properly and cocompactly on a CAT(0) cube complex.
\end{theorem}

After Perelman's proof of Thurston geometrization conjecture, one of the remaining conjectures was the virtual Haken conjecture (stated by Waldhausen): \textit{Any closed and aspherical 3-dimensional manifold has a finite cover which is Haken}. A manifold $M$ is \emph{Haken} if it contains an embedded surface $S$ such that the inclusion $S\hookrightarrow M$ induces an embedding between fundamental groups $\Pi_1(S)\hookrightarrow\Pi_1(M)$.

Building on Theorem \ref{KM} and the theory of special cube complexes \cite{MR2377497}, Agol was able to show that any hyperbolic group that acts properly cocompactly on a CAT(0) cube complex is virtually special. Eventually, he proved the virtual Haken conjecture. For details, see \cite{MR2986461}.
\begin{theorem}[{\cite{MR3104553}}]Every hyperbolic 3-manifold is virtually Haken.
\end{theorem}
\section{Rigidity properties of the above examples}\label{rigidity}
\subsection{Bieberbach theorems and the solvable subgroup theorem}
We aim to show in that section how the above examples are rigid. Let us start with the simplest example given by Euclidean spaces. A group acting discretely cocompactly by isometries on $\R^n$ is called a \emph{crystallographic group}\index{crystallographic group}\index{group!crystallographic} of $\R^n$.
\begin{theorem}[Bieberbach theorems]\label{Bieberbach}\index{Bieberbach theorem}\index{theorem!Bieberbach} Any crystallographic group of  $\R^n$ contains an Abelian group $\Z^n$ of finite index. This Abelian subgroup is generated by $n$ linearly independent translations.

There are finitely many isomorphism classes of crystallographic groups of  $\R^n$.
\end{theorem}
Accounts on those theorems may be found in \cite{MR0185012,MR798909} and \cite[\S 4.2]{MR1435975}. Theorem \ref{Bieberbach} implies the existence of a free Abelian subgroup acting properly by semi-simple isometries. One has the following converse. 

\begin{theorem}[Flat torus theorem]\index{flat torus theorem}\index{theorem!flat torus} Let $A$ be a free Abelian group of rank $n$ acting properly by semi-simple isometries on a CAT(0) space $X$. Then there is a closed convex subspace $Y\subset X$ which is $A$-invariant and splits as a product $Z\times\R^n$. The action of $A$ on $Y$ is diagonal, being trivial on $Z$ and by translations on $\R^n$. 
\end{theorem}

From this first rigidity result, one can show that solvable subgroups of groups acting geometrically on CAT(0) spaces are very restricted. A group has \emph{virtually}\index{virtually} some property if there is a subgroup of finite index which has the property.

\begin{theorem}[Solvable subgroup theorem]\index{solvable subgroup theorem}\index{theorem!solvable subgroup} Let $\Gamma$ be CAT(0) group. Then any solvable subgroup of $\Gamma$ is virtually an Abelian free group of finite rank.
\end{theorem}
\subsection{Superrigidity}\label{super}

In this section, we are interested in rigidity of lattices (discrete subgroups of finite covolume) of semi-simple algebraic groups over local fields. Recall that a local field is a topological field which is locally compact. There is a short list of such fields: finite fields $\mathbb{F}_q$, $\R$, $\C$, finite extensions of $\mathbb{Q}_p$ ($p$-adic numbers) and finite extension of $\mathbb{F}_q((T))$ (Laurent series over a finite field).

The idea here under the term \emph{rigidity}\index{rigidity} is the fact that a lattice determines the group where it lives. Thanks to the dictionary between semi-simple algebraic groups over local fields and their symmetric spaces or Euclidean buildings, we state all results using the action on the associated CAT(0) space.

\begin{theorem}[{Mostow Strong Rigidity \cite{MR0385004}}]\label{Mostow}Let $X,Y$ be two symmetric spaces and $\Gamma$ be a group acting freely cocompactly by isometries on both $X,Y$ via two homomorphisms $\alpha\colon\Gamma\to\Isom(X)$ and $\beta\colon\gamma\to\Isom(Y)$. If $X$ has no factor isometric to $\mathbb{H}^2$ then there is an isomorphism $\varphi\colon\Isom(X)\simeq\Isom(Y)$ such that $\beta=\varphi\circ\alpha$.
\end{theorem}

Let us explain why the hyperbolic plane has to be excluded. If $S$ is a Riemann surface of genus at least 2, it can be endowed with a hyperbolic metric and thus its universal cover is $\mathbb{H}^2$ (this is a \emph{uniformization} of the surface). But there are many non-isometric ways of endowing $S$ with a hyperbolic metric (the different possibilites are parametrised by the Teichm\"uller space of $S$). Two non-isometric ways give two non-conjugate embeddings of $\Pi_1(S)$ in $\mathbb{H}^2$ as a lattice.

A few years after Mostow proved his rigidity theorem, Margulis proved a very stronger result  \cite{MR1090825}. The name \emph{superrigidity}\index{superrigidity} for that result was coined by Mostow. Margulis stated the result purely in group-theoretic terms. It was an intermediate result to obtain arithmeticity of lattices. We use a more geometric rephrasing \cite{MR1168043,MR2655318}.
\begin{theorem}[Geometric Superrigidity] Let $X$ and $Y$ be proper symmetric CAT(0)  spaces or Euclidean buildings of rank at least 2. If $\Gamma$ is a lattice of $\Isom(X)$ and $\Gamma$ acts by isometries on $Y$ without fixed point in $\overline Y$ then there is a closed convex subspace $Z\subset Y$ which is isometric to a product of irreducible factors of $X$ and the action $\Gamma\action Z$ factorises via a quotient of $\Isom(X)$.
\end{theorem}

A similar result for symmetric spaces of rank 1 which are not the real or complex hyperbolic space were obtained later \cite{MR1147961,MR1215595}. This superrigidity result inspired a lot of different works in many different directions. Let us just mention \cite{MR2219304} for lattices in products, \cite{MR3343349} for symmetric spaces of infinite dimension and finite rank and \cite{MR2377496} for Busemann non-positively curved targets.

Let us introduce a weakening of being isometric for two metric spaces. Roughly speaking, two metric spaces are quasi-isometric if they are homothetic at large scale.
\begin{definition} Let $(X,d)$ and $(X',d')$ be two metric spaces. A map $f\colon X\to X'$ is a \emph{quasi-isometry}\index{quasi-isometry} if there are constants $L,C>0$ such that for any $x,y\in X$,
$$1/L\, d(x,y)-C\leq d'(f(x),f(y))\leq Ld(x,y)+C$$
and for any $x'\in X'$ there is $x\in X$ with $d'(x',f(x))\leq C$.
\end{definition}
It is a simple computation to show that if $f\colon X\to X'$ is a quasi-isometry then there is $g\colon X'\to X$ which also is a quasi-isometry such that for any $x\in X$, $d(x,g(f(x)))\leq 2LC$. In that sense $g$ is a \emph{quasi-inverse} to $f$.

In Mostow's theorem (Theorem \ref{Mostow}), the two actions of the common cocompact lattice gives an \emph{equivariant} quasi-isometry between the two symmetric spaces $X,Y$. Mostow showed that this quasi-isometry lies at bounded distance from an isometry (up to rescaling the metric on each factor). Margulis conjectured that if $X,Y$ are irreducible symmetric spaces of higher rank then the conclusion should also hold without the equivariance of the quasi-isometry.

\begin{theorem}[{\cite[Theorem 1.1.3]{MR1608566}}] Let $X$ be an irreducible symmetric space of higher rank or an irreducible higher rank Bruhat-Tits building and let $Y$ be a product of Euclidean spaces, symmetric spaces and Euclidean thick buildings. If $X$ and $Y$ are quasi-isometric then $Y$ has only one factor which is  isometric (after rescaling) to $X$ and the quasi-isometry is at bounded distance from that isometry.
\end{theorem}

Kleiner and Leeb proved that for products of Euclidean spaces, symmetric spaces and Euclidean buildings, the splitting as a product of irreducible factors is an invariant of quasi-isometry. In particular, they were able to prove the Margulis conjecture. 

\begin{corollary}\label{kl}If $X$ is an irreducible higher rank symmetric space and $Y$ is a quasi-isometric symmetric space  then $X$ and $Y$ are homothetic.
\end{corollary}

The symmetric spaces of rank 1 are hyperbolic spaces over $\R,\C$, the quaternions or Cayley octonions. In the last two cases, Corollary \ref{kl} holds but for the first two ones, there are much more quasi-isometries than the ones at bounded distance from homotheties.

\subsection{Isometries of proper CAT(0) spaces}
Semi-simple algebraic groups over local fields (and their subgroups) are leading examples of groups acting on proper CAT(0) spaces. As soon as the CAT(0) space is proper, the isometry group is locally compact. The structure theory of locally compact groups is quite well understood and in particular, the solution of the fifth Hilbert   problem tells us when a connected locally compact group is actually a Lie group and, in particular, the connected component of the isometry group of a symmetric space.

In a series of two papers \cite{MR2574741,MR2574740}, Caprace and Monod developed a structure theory for groups acting on proper CAT(0) spaces. The settings are cocompact actions of groups acting on a CAT(0) space or merely that the action of the full group of isometry is cocompact. The two motivating questions are: What are the general properties of these pairs of groups and spaces? How to characterise Euclidean buildings and symmetric spaces? One main theme underlying this question is splitting canonically the space into irreducible pieces (de Rham decompositions).

They prove a lot of theorems and we emphasise just a few of them. These may not be the strongest ones but they illustrate well the general ideas. The first one can be seen as a solution of the fifth Hilbert  problem for CAT(0) spaces. Recall that a locally  compact group $G$ has a connected component of the identity $G^0$ which is a normal closed subgroup such that $G/G^0$ is totally disconnected. Moreover $G^0$ has a unique maximal normal amenable subgroup $A$ (the amenable radical) such that $G^0/A$ is a semi-simple Lie group without compact factors. 

A CAT(0) space is \emph{geodesically complete}\index{geodesically complete space}\index{space!geodesically complete} if any segment is included in some geodesic line. Under geodesic completeness, understanding the relation between splitting and these normal subgroups leads to the following theorem.

\begin{theorem}[{\cite[Theorem 1.1]{MR2574740}}] Let $X$ be a proper geodesically complete CAT(0) space such that $\Isom(X)$ acts cocompactly without fixed point at infinity. Then $X$ splits as a product
$$X\simeq M\times \R^n\times Y$$
where $M$ is a symmetric space and $\Isom(Y)$ is totally disconnected and acts by semi-simple isometries. 

Moreover, this splitting is preserved by isometries.
\end{theorem}

Under a supplementary assumption on stabilisers of points at infinity, one gets the following characterisation of symmetric spaces and Euclidean buildings.

\begin{theorem}[{\cite[Theorem 1.3]{MR2574740}}]Let $X$ be a geodesically complete proper CAT(0) space. Suppose that the stabiliser of every point at infinity acts cocompactly on $X$.
Then $X$ is isometric to a product of symmetric spaces, Euclidean buildings and Bass--Serre trees.
\end{theorem}

Discrete subgroups of semi-simple algebraic groups over local fields have been studied extensively (see e.g. \cite{MR1090825}). In particular, for lattices, that is, discrete subgroups of finite covolume,  two important theorems are the Borel density theorem and the Margulis arithmeticity theorem. Roughly speaking, the first one  proves Zariski density of lattices and the second proves arithmeticity of lattices in higher rank.

Since isometry groups of proper CAT(0) spaces have Haar measures, one has also a notion of lattice.
 
\begin{definition} Let $X$ be a proper CAT(0) space with cocompact isometry group. A \emph{CAT(0) lattice}\index{CAT(0)!lattice}\index{lattice!CAT(0)} is a lattice of $\Isom(X)$.
\end{definition}
The Borel density theorem implies that any lattice of a connected  semi-simple real algebraic group without compact factor is Zariski dense. For connected semi-simple Lie group without compact factor, Zariski density of a subgroup is understood via the action on the associated symmetric space. Let $X$ be a symmetric space. A subgroup $\Gamma\leq\Isom(X)^0$ is Zariski-dense if and only if $\Gamma\action X$ is minimal and $\Gamma$ does not fix any point at infinity. If the last two conditions are satisfied, one says that the action is \emph{geometrically dense}\index{geometrically dense action}\index{action!geometrically dense}.

Here is an analog for general proper CAT(0) spaces.

\begin{theorem}[{\cite[Theorem 1.1]{MR2574741}}]Let $X$ be a proper CAT(0) space, $G$ a locally compact group acting continuously by isometries on X and $\Gamma < G$ a lattice. Suppose that $X$ has no Euclidean factor.
If $G$ acts minimally on $X$ without fixed point at infinity, so does $\Gamma$.
\end{theorem}

Surprisingly, the mere existence of a parabolic isometry in the full isometry group can lead to a very strong result like the following arithmeticity statement.
\begin{theorem}[{\cite[Theorem 1.5]{MR2574741}}]Let $(\Gamma,X)$ be an irreducible finitely generated CAT(0) lattice where $X$ is geodesically complete and possesses some parabolic isometry.
If $\Gamma$ is residually finite, then $X$ is a product of symmetric spaces and Bruhat-Tits buildings. In particular, $\Gamma$ is an arithmetic lattice unless $X$ is a real or complex hyperbolic space.
\end{theorem}

\section{Amenability and non-positive curvature}

\subsection{Amenability}

Recall that a topological group is \emph{amenable} if for any compact $G$-space (in a locally convex space on which $G$ acts affinely), there is a fixed point. First examples of amenable groups are given by Abelian and compact groups and amenability is stable under group extension and passing to closed subgroups. For example, virtually solvable groups are amenable. For a reference on amenability we recommend  \cite[Appendix G]{MR2415834} which is concise and contains full proofs of the equivalence between the most useful definitions of amenability. 

One main obstruction to amenability is the existence of a closed free subgroup. Free groups are not amenable because of paradoxical decompositions (See for example \cite{MR803509}).

In the realm of CAT(0) spaces, the first occurence of amenability is due to flatness. Recall that the isometry group of the Euclidean space $\R^n$ is $\Or(n)\ltimes\R^n$ which is amenable as an extension of a compact group by an Abelian one.

Since linear groups are one of our main sources of groups acting on spaces of non-positive curvature, it is natural to concentrate first on that case. 
\begin{theorem}[{Tit's Alternative \cite[Theorem 1]{MR0286898}}]\index{Tits' alternative} Let $\Gamma$ be a subgroup of $\GL(n,k)$ where $k$ is a field of characteristic 0 and $n\in\N$. Then $\Gamma$ is virtually solvable or contains a non-Abelian free subgroup.
\end{theorem}

In positive characteristic, there is a slightly more general statement \cite[Theorem 2]{MR0286898}. This alternative has many interesting consequences. For example, any finitely generated linear group (in any characteristic) contains a free group or is virtually solvable. In particular, it has exponential growth or polynomial growth.

\begin{corollary}
A finitely generated amenable linear group is virtually solvable.
\end{corollary}

The original proof uses the dynamics of the action of linear groups on the projective space where a ping-pong argument produces non-Abelian free groups. So it does not use the action of $\Gamma$ on the symmetric space or the Euclidean building (when $k$ is a valued field) $X$ associated to $\GL(n,k)$. 

Because of the importance of the Tits alternative for linear groups, it is natural to try to extend the result to other classes than linear groups. For example, there is a version of the alternative for hyperbolic groups \cite[Th\'eor\`eme 38]{MR1086648}. See the introduction of \cite{MR2164832} for other classes. Let us give two examples where the action on some CAT(0) space leads to an analog of Tits' alternative. We start with groups acting on CAT(0) cube complexes (see also \cite{MR2827012,Fernos:2015rz}).

\begin{theorem}[{\cite{MR2164832}}]\label{Titscc} Let $G$ be a group with a bound on the orders of finite subgroups. If $G$ acts properly on a finite-dimensional  CAT(0) cube complex then either $G$ contains a non-Abelian free subgroup or $G$ is a finitely generated virtually Abelian group.
\end{theorem}

The finite dimension assumption on the complex is very important. For example, the Thompson group $F$ has no non-Abelian free subgroups, is not virtually solvable but acts on a proper (infinite-dimensional) CAT(0) cube complex \cite{MR2393179}. It is an open problem to decide wether the Thompson group $F$ is amenable or not.

Another example is given by \cite[Th\'eor\`eme C]{MR2811600} which shows that the group of birational transformations of a compact K\"ahler complex surface satisfies the Tits alternative. This means that any finitely generated subgroup contains a non-abelian free subgroup or is virtually solvable. The proof relies on the construction of an action of the group of birational transformations on the infinite-dimensional hyperbolic space.

Since the Tits alternative holds  for linear groups and groups acting on CAT(0) cube complexes, which are two main sources of groups acting on CAT(0) spaces, it is natural to try to know if it holds more generally (see \cite[Problem 1.3]{MR2164832} and \cite[Introduction]{MR3072802}).

\begin{question} Is there a version of the Tits Alternative for groups acting on proper cocompact CAT(0) spaces?
\end{question}

After the flat torus and solvable subgroup theorems, the main tool to understand amenable groups is the Adams-Ballmann Theorem which generalises previous results for manifolds of non-positive curvature. See \cite[Theorem 2]{MR875344} and references in \cite{MR1645958}.
\begin{theorem}[{\cite{MR1645958}}]Let $X$ be a proper CAT(0) space. If $G$ is an amenable group acting continuously by isometries on $X$ then $G$ stabilises a flat subspace (possibly of dimension 0, that is a point) or fixes a point at infinity.
\end{theorem} 

The theorem has also been  proved when the space has finite telescopic dimension,  instead of being proper \cite{MR2558883}. Observe that the conclusion cannot hold without any assumption on the CAT(0) space. Edelstein exhibits an isometry of a Hilbert space of infinite dimension with no fixed point at infinity nor invariant (finite-dimensional) Euclidean subspace \cite[\S 2]{MR0164222}.

The Tits alternative gives a complete description of amenable linear groups. For sufficiently homogenous CAT(0) spaces, a description of amenable subgroups has been obtained. A locally compact group is \emph{locally elliptic}\index{locally elliptic!group}\index{group!locally elliptic} if it is an increasing union of compact subgroups. A locally compact group $H$ has a unique maximal normal locally elliptic subgroup, its \emph{locally elliptic radical}\index{locally elliptic!radical}. Let us denote it by $R(H)$.

\begin{theorem}[{\cite[Theorem A]{MR3072802}}]Let $X$ be a proper cocompact CAT(0) space. A closed subgroup $H\leq\Isom(X)$ is amenable if and only if the following equivalent conditions hold:
\begin{enumerate}
\item The identity component $H^\circ$ is solvable by compact,
\item $H^\circ R(H)$ is open in $H$,
\item $H/(H^\circ R(H))$ is virtually solvable.
\end{enumerate}
\end{theorem}

This algebraic statement has also a geometric counterpart. To state it, let us introduce the transverse space to a point at infinity. Let $\xi$ be a point at infinity of some CAT(0) space $X$. One introduces a pseudo-metric on the geodesic rays $\rho$ such that $\rho(+\infty)=\xi$. For two such rays $\rho_1,\rho_2$,
$$d(\rho_1,\rho_2)=\inf_{t,s>0}d(\rho_1(t),\rho_2(s)).$$
After identifying rays at distance 0, one gets a metric space whose completion $X_\xi$ is a new CAT(0) space on which the stabiliser of $\xi$ in $\Isom(X)$ acts by isometries. This space is called the \emph{transverse space}\index{space!transverse}\index{transverse space} associated to $\xi$ \cite{MR1934160,MR2495801}.  

Let us illustrate this concept in some simple cases. If $X$ is a Euclidean space, a point at infinity is given by parallel geodesic rays, and two rays $\rho_1,\rho_2$ satisfy $d(\rho_1,\rho_2)=0$ if and only if there is $t_0\in\R$ such that $\rho_1(t)=\rho_2(t+t_0)$ for $t$ large enough. In particular, the transverse space to $\xi\in\partial X$ can be identified with an orthogonal space to some geodesic ray in the class of the point $\xi$. If $X$ is the hyperbolic plane, any two geodesic rays $\rho_1,\rho_2$ with a common endpoint at infinity satisfy $d(\rho_1,\rho_2)=0$ and the transverse associated to any point at infinity is reduced to a point.

For symmetric spaces, the construction goes back to Karpelevich \cite{MR0231321} and can be expressed in the following way. If $\xi$ is a point at infinity of a symmetric space $X$, the stabiliser of $\xi$ is a parabolic group $P_\xi$ and the transverse space $X_\xi$ is the product of a Euclidean space and another symmetric space which is the symmetric space associated to the semi-simple part of $P_\xi$ in a Levi decomposition (see \cite[\S2.217]{MR1441541} for the background required). The space $X_\xi$ can be isometrically embedded in $X$. If $\xi'\in\partial X$ is opposite to $\xi$ (that means $\xi,\xi'$ are the ends of a common geodesic line) then the union of geodesic lines with ends $\xi$ and $\xi'$ is a totally geodesic subspace which splits as $\R\times X_\xi$. The $\R$-factor is given by any line joining $\xi$ and $\xi'$. 

In that case the space $X_\xi$ strongly depends on $\xi$ (actually on the minimal cell of the spherical building at infinity containing it). For example, if $P_\xi$ is a  minimal parabolic then $X_\xi$ is a Euclidean space of dimension $\rank(X)-1$. Let $G$ be any semi-simple group of non-compact type. Then $G$ can be embedded (via the adjoint action) in $\SL_n(\R)$ for $n$ large enough and by the Mostow-Karpelevich  theorem \cite{MR0069829}, the symmetric space of $G$ embeds isometrically in the one of $\SL_n(\R)$. Actually, it lies as a factor in some transverse space at infinity (as soon as $G\neq\SL_n(\R)$).

This construction of a transverse space can be repeated and by induction, one constructs $X_{\xi_1,\dots,\xi_i}$ where $\xi_i\in\partial X_{x_1,\dots,x_{i-1}}$. A \emph{refined flat}\index{refined flat}\index{flat!refined} of $X$ is a flat in some $X_{\xi_1,\dots,\xi_i}$. With this refined construction, one has a geometric characterisation of amenable subgroups of $\Isom(X)$.

\begin{theorem}[{\cite[Corollary G]{MR3072802}}]Let $X$  be a proper cocompact CAT(0) space. A closed subgroup of $\Isom(X)$ is amenable if and only if it preserves a refined fat.
\end{theorem}

\subsection{Amenability at infinity}

Even if a group is not amenable, it may have \emph{amenable actions}\index{action!amenable}\index{amenable!action}. There are different notions of amenable actions. Here is the one we use, which is sometimes called \emph{topological amenability}\index{topological amenability}\index{amenability!topological}. The space $\Prob(G)$ of probability measures on a Hausdorff locally compact group $G$ is endowed with the weak*-topology as a subspace of the dual space of continuous functions with compact support on $G$ (see e.g. \cite[Chapter 7]{MR1681462}).

\begin{definition}\label{defamena} Let $B$ be a locally compact space and $G$ a locally compact group acting continuously by homeomorphims on $B$. The action $G\action B$ is \emph{amenable} if there is a sequence of continuous maps $\mu_n\colon B\to \Prob(G)$ such that
$$\|\mu_n(gb)-\mu_n\|\to0$$
uniformly on compact subsets of $G\times B$.
\end{definition}
The norm used here is the dual norm on $C_c(G)^*$. This coincides with the norm of total variation. 

The case we are interested in is when $B$ is compact. We refer to \cite{MR1926869,MR1799683}.
\begin{definition} A locally compact group $G$ is \emph{amenable at infinity}\index{amenable!at infinity}\index{group!amenable at infinity} if there is a compact space $B$ and an amenable $G$-action on $B$.
\end{definition}
For example the action of an amenable group on a point is amenable and thus amenable groups are amenable at infinity. For semi-simple algebraic groups over local fields $G$, a minimal parabolic group is amenable and $G/P$ is compact, thus $G$ is amenable at infinity. The amenability is \emph{at infinity} because $G/P$ can be identified with the set of maximal simplices of the spherical building at infinity. For a survey on this specific subject one can read \cite{MR2275659}. For discrete groups, amenability at infinity is equivalent to exactness and Property A. 
Another example is given by groups acting on trees: The action on the boundary at infinity is amenable. This is the origin of examples of amenability at infinity that we see in the remaining of this section. In rank less than 3, Euclidean buildings are not necessarily of algebraic origin but one still has amenability at infinity.

\begin{theorem}[{\cite{MR2680274}}]Let $X$ be a locally finite building. The action of $\Aut(X)$ on the combinatorial compactification of $X$ is amenable.
\end{theorem}

We denote by $\partial_R X$ the Roller boundary of a CAT(0) cube complex $X$. In \cite{MR2490224} it is proved that discrete groups acting metrically properly on a finite-dimensional CAT(0) cube complex are amenable at infinity. Actually the proof leads to the following slightly more general statement.

\begin{theorem}\label{amenableatinfinity}Let $X$ be a finite-dimensional CAT(0) cube complex and $G$  a discrete group acting on $X$ by automorphisms with finite vertex stabilisers. The action of $G$ on the compact space $\partial_R X$  is amenable.
\end{theorem}

\begin{proof}We use the notations of \cite{MR2490224} and we adapt the proof of Theorem 4.2 there. First we make the notation of $f_{n,x}$ dependent of $z\in \partial_R X$. Thus we denote by $f_{n,x}^z$ the function defined by

$$f_{n,x}^z(y)=\begin{cases}
\binom{n-d(x,y)+\delta_z(y)}{\delta_z(y)}&\textrm{if } y\in[x,z]\\
0&\textrm{if } y\notin[x,z]
\end{cases}$$
where $\delta_z(y)=N-|\mathfrak{N}_z(y)|$ ($N$ is the dimension of the complex). As in \cite{MR2490224}, this function is finitely supported and thus $\ell^1$. For $g\in\Aut(X)$, the equivariance relation is $g\cdot f_{n,x}^z=f^{gz}_{n,gx}$. It is proved that 

$$\frac{\|f_{n,x}^z-f_{n,x'}^z\|}{\|f_{n,x}^z\|}\leq\frac{2d(x,x')N}{n+N}.$$

Moreover, it is shown that if $z_i\to z$ then $f^{z_i}_{n,x}$ coincides with $f^{z}_{n,x}$ for $i$ large enough and we get the following continuity property: for fixed $n,x$, the map $z\mapsto f^z_{n,x}$ is locally constant. 

Now, define $\phi^z_n=\sum_{x\in T} \frac{f^z_{n,x_0}(gx)}{|G_x|}$ where $T$ is a transversal of $G\action X^0$ ($X^0$ is the vertex set of $X$) and $G_x\leq G$ the isotropy group of $x$. One has that $\phi^z_n$ is finitely supported and $\|\phi^z_n\|=\|f_{n,x_0}^z\|$ (which actually does not depend on $z$ nor on $x_0$).

For $z\in\partial_R X$ and $g\in G$,

\begin{align*}
\|g\cdot\phi_n^z-\phi_n^{gz}\|&=\sum_{h\in G}|\phi_n^z(g^{-1}h)-\phi^z_n(h)|\\
&\leq\sum_{h\in G,\ x\in T}\frac{|f_{n,x_0}^z(g^{-1}hx)-f^z_{n,x_0}(hx)|}{|G_x|}\\
&\leq \sum_{h\in G,\ x\in T}\frac{|f_{n,gx_0}^{gz}(hx)-f^z_{n,x_0}(hx)|}{|G_x|}\\
&\leq \sum_{x\in X}|f_{n,x_0}^z(x)-f_{n,x_0}^{gz}(x)|\\
&\leq\|f_{n,x_0}^z-f_{n,x_0}^{gz}\|\\
&\leq 2\,d(x_0,gx_0)\frac{N}{n+N}\|\phi_n^z\|
\end{align*}

Defining $m^z_n=\frac{	\phi^z_n}{\|\phi^z_n\|}$, all properties of the Definition \ref{defamena} are satisfied and $G\action \partial_R X$ is amenable.
\end{proof}

%\begin{question} Is the definition of $\phi_n$ independant of $T$? If yes, can one define $\phi_n$ without reference to a transversal?
%\end{question}

More generally in the spirit of \cite[Corollary 2.9]{MR2275659} (and also \cite{MR2243738}):

\begin{theorem}\label{amenableatinfinity2}Let $X$ be a finite-dimensional CAT(0) cube complex and $G$  a countable group acting on $X$ by automorphisms with amenable vertex stabilisers. The action of $G$ on the compact space $\partial_R X$  is amenable.
\end{theorem}

\begin{proof} Up to considering a $G$-invariant convex subcomplex of $X$, we assume that $X$ has countably many vertices. We directly use \cite[Proposition 2.7]{MR2275659} where the countable $G$-space is the space of vertices of $X$ and the compact $G$-space is the Roller boundary $\partial_R X$. Fix a base vertex $x_0$ and for $z\in\partial_R X$, define $\mu_n^z$ to be the probability measure $f_{n,x_0}^z/\|f_{n,x_0}^z\|$. A similar computation as in the proof of Theorem \ref{amenableatinfinity} shows that $\|g\mu_n^z-\mu_n^{gz}\|\leq 2\,d(x_0,gx_0)\frac{N}{n+N}$ and thus we can apply directly \cite[Proposition 2.7]{MR2275659} with $Y$ being a point (which is a $G$-space with trivial action and an amenable space for any stabiliser of point in $X$). We get that $G\action \partial_R X$  is amenable.
\end{proof}
\begin{corollary}\label{amenhigman}The Higman group acts amenably on the Roller boundary of its square complex.
\end{corollary}

\begin{remark} The Higman group is known to be exact because it has finite asymptotic dimension. Corollary \ref{amenhigman} yields another proof of this fact. Observe that \cite{MR2490224} is not sufficient because the action of the Higman group on its complex is not metrically proper (vertices have infinite stabilisers).
\end{remark}

%Arguing similarly, one can prove the following.
%\begin{theorem}Let $X$ be a finite-dimensional CAT(0) cube complex and $G$ be a countable group acting on $X$ by automorphisms. If vertex stabilizers are exact then $G$ is exact.
%\end{theorem}
\subsection{Haagerup property and property (T)}\label{Haagerup}

There is an another weakening of amenability related to actions on some  CAT(0) spaces: the Haagerup property.

\begin{definition}A locally compact group has the \emph{Haagerup property}\index{Haagerup!property}\index{property!Haagerup} if it admits a continuous and metrically proper action by isometries on a Hilbert space.
\end{definition}

Any continuous action of a group $\Gamma$ on a Hilbert space $\mathcal{H}$ is given by a continuous homomorphism $\alpha\colon \Gamma\to\Or(\mathcal{H})$ and a continuous cocycle $b\colon\Gamma\to \mathcal{H}$ (that is $b(gh)=\alpha(g)b(h)+b(g)$ for all $g,h\in\Gamma$). The action is \emph{metrically proper}\index{action!metrically proper}\index{metrically proper action} if $\| b(g)\|$ goes to infinity when $g$ leaves all compacta of $\Gamma$.

With a convenient characterisation of amenability, it is not difficult to show that amenable groups have the Haagerup property \cite{MR1388307}. Originally, the Haagerup property was defined in terms of mixing unitary representations. It was shown that for second countable groups, the Haagerup property is equivalent to Gromov's a-(T)-menability, that is, the definition we gave. For a panorama and characterisations of groups with the Haagerup property see \cite{MR1852148}. 

The name that Gromov gave to this property comes from the strong opposition between the Haagerup property and property (T). Actually, a group with both properties is compact.

\begin{definition} A second countable locally compact group has \emph{property (T)}\index{property!(T)} if any continuous isometric action on a Hilbert space has fixed points.
\end{definition}

A good reference about property (T) is \cite{MR2415834}. In the algebraic world, the situation is well understood. 

\begin{theorem}\label{HaaT} Let $k$ be a local field and $G$ be the $k$-points of a connected almost simple algebraic group $k$. Let $r$ be the $k$-rank of $G$.
\begin{itemize} 
\item If $r\geq2$ then $G$ has property (T) (\cite[Theorem 1.6.1]{MR2415834}). 
\item If $r=1$ and $k$ is non-Archimedean then $G$ has the Haagerup property.
\item If $r=1$ and $k=\C$ then $G$ is locally isomorphic to $\PSL_2(\C)\simeq \SO^+(3,1)$ and has the Haagerup property.
\item If $r=1$ and $k=\R$, the two properties occur: $\SO(n,1)$ and $\SU(n,1)$ have the Haagerup property but $\Sp(n,1)$ and $\F_4^{-20}$ have property (T). 
\end{itemize}
\end{theorem}

This theorem shows that knowing that $G$ acts on a symmetric space or a Euclidean building (or more generally on a CAT(0) space) does not suffice to conclude it has one or the other property. Many groups with property (T) act on CAT(0) polyhedral cell complexes without fixed point (see for example \cite{MR1946553,MR1465598}) although we will characterise this property via fixed points on some CAT(0) spaces. 

Nonetheless, one can explain how the Haagerup property occurs in Theorem \ref{HaaT}. If the rank is 1 and $k$ is non-Archimedean, $G$ acts properly on the associated  Euclidean building which is a tree, thus a space with walls. The groups $\SO(n,1)$ act on the real hyperbolic space of dimension $n$ which has a natural structure of a space with measured walls. Walls are given by totally geodesic hypersurfaces and a Crofton formula shows that the distance between two points is proportional to the measure of walls separating them \cite[Proposition 2.6.4]{MR2415834}. The role of spaces with measured walls is explained below.

The definitions of the two properties involve Hilbert spaces which are completely flat. Nevertheless these properties can be characterised by actions on negatively curved spaces.

\begin{proposition} Let $X$ be the real or complex hyperbolic space of infinite dimension and let $\Gamma$ be a second countable locally compact group.

\begin{itemize}
\item $\Gamma$ has the Haagerup property if and only if $\Gamma$ has a metrically proper action by isometries on $X$.
\item $\Gamma$ has property (T) if and only if any action of $\Gamma$ on $X$ has fixed points.
\end{itemize}
\end{proposition}
The two properties are stable under passing to lattices: if $\Gamma$ is a lattice in $G$ then $G$ has the Haagerup property (respectively property (T)) if and only if $\Gamma$ has the same property (see \cite[Theorem 1.7.1]{MR2415834} and \cite[Proposition 6.1.5]{MR1852148}). But neither the Haagerup property nor property (T)  are geometric, that is, invariant under quasi-isometries. See \cite{Carette:2014qy} and \cite[3.6]{MR2415834}. Nonetheless, one can give combinatorial/geometric characterisations using actions on \emph{spaces with measured walls}. We have seen what is a space with walls where points are separated by a finite number of walls. There is a generalisation of this notion where there is a measure on the set of walls and points are separated by a set of walls of finite measure.

\begin{theorem}[{\cite[Theorem 1.3]{MR2671183}}]Let $G$ be a locally compact second countable group.
\begin{enumerate}
\item The group $G$ has property (T) if and only if any continuous action by automorphisms on a space with measured walls has bounded orbits.
\item The group G has the Haagerup property if and only if it admits a proper continuous action by automorphisms on a space with measured walls.
\end{enumerate}
\end{theorem}

Let us explain a bit this theorem in a simple case: groups acting on trees. If a group $G$ acts on a tree $T$ then one gets an isometric action on a associated Hilbert space $\mathcal{H}$. If $G$ has a fixed point on $\mathcal{H}$ then $G$ has a fixed point in $T$. If the action on $T$ is metrically proper then the action on $\mathcal{H}$ is also metrically proper.

Let $E$ be the set of oriented edges of $T$ and let $\mathcal{H}=L^2(E)$, that is, the Hilbert space with orthonormal basis $\{\delta_e\}_{e\in E}$. If $G$ acts on $E$, it permutes edges and thus one gets an orthogonal representation of $G$ on $\mathcal{H}$. Fix a vertex $v_0$ and for any vertex $v$ denote by $[v,v_0]$ the finite set of edges between $v_0$ and $v$. Now define $b(g)=\sum_{e\in[gv_0,v_0]}\varepsilon_g(e)\delta_e$ where $\varepsilon_g(e)=\pm1$ is positive if $e$ points from $gv_0$ to $v_0$  and negative otherwise. It is not difficult to check that $b$ is actually a cocycle and that the associated isometric action on $\mathcal{H}$ has the above properties.

If one removes an edge from a tree one gets two connected components and thus a partition of the set of vertices in two parts. So the set of unoriented edges yields a structure of a space with walls on the set of vertices of a tree. The distance between two vertices is exactly the number of walls separating them.

One can generalise this construction from trees to spaces with (measured) walls. The relevant Hilbert space is the $L^2$-space on half-spaces.
\section{Rank rigidity and the flat closing conjectures}\label{rrc}
\subsection{Rank Rigidity Conjecture}

A \emph{hyperbolic isometry}\index{isometry!hyperbolic}\index{hyperbolic!isometry} is a semi-simple isometry with positive translation length.
\begin{definition} Let $X$ be a CAT(0) space. A \emph{rank one} isometry is a hyperbolic isometry such that none of its axes is the boundary of an isometrically embedded Euclidean half-plane.
\end{definition}

Rank one isometries are very interesting because of the dynamics at infinity they induce. There is a so-called \emph{north-south dynamics}. If $g$ is a rank-one isometry and $X$ is proper then $g$ has exactly two fixed points at infinity $\xi_-,\xi_+\in\bo X$ and for any pair of neighbourhoods $V_\pm$ of $\xi_\pm$, there is $k\in\N$ such that $g^k(\bo X\setminus V_-)\subset V_+$ \cite[Lemma 4.4]{MR2581914}.

\begin{conjecture}[Rank Rigidity Conjecture \cite{MR2478464}] Let $X$ be an irreducible geodesically complete and proper CAT(0) space. If $\Gamma$ is a countable group acting properly and cocompactly on $X$ then $X$ is a symmetric space of non-compact type and rank at least 2, a Euclidean building of dimension at least 2 or $\Gamma$ contains a rank one isometry.
\end{conjecture}

This conjecture comes from the following theorem for Hadamard manifolds.

\begin{theorem}[Rank Rigidity Theorem \cite{MR908215,MR819559}]\index{rank rigidity theorem}\index{theorem!rank rigidty}Let $X$ be an irreducible Hadamard manifold with a countable group $\Gamma$ acting properly and cocompactly by isometries. Then, either $M$ is 	a symmetric space of non-compact type and rank at least 2 or $\Gamma$ contains a rank one isometry.
\end{theorem}

For a more detailed discussion about the rank rigidity theorem, we refer to \cite[Theorem C]{MR1377265}. Beyond the world of manifolds, the conjecture has been proved for a large class of singular CAT(0) spaces.

\begin{theorem}[Rank rigidity theorem for CAT(0) cube complexes {\cite[Theorem A]{MR2827012}}] Let $X$ be a finite-dimensional CAT(0) cube complex and $\Gamma$ a group acting minimally by isometries on $X$. If $\Gamma$ has no fixed point in $\overline{X}$ then $X$ splits as the product of two cube sub-complexes or $\Gamma$ contains a rank one isometry.
\end{theorem}

\begin{remark} If one removes the minimality hypothesis, then the lack of fixed points implies the existence of a convex $\Gamma$-invariant subcomplex $Y\subset X$ such that the action $\Gamma\action Y$ is minimal. If $Y$ does not split as a product then $\Gamma$ contains an element which acts as a rank one isometry on $Y$.
\end{remark}

Having the rank rigidity conjecture in mind, the rank one case appears to be the generic case  and thus it is natural to try to understand rank one isometries on one side \cite{MR2581914} and the possibility for spaces to have rank one isometries \cite{MR2961285}.

\subsection{Flat Closing Conjecture}

There is another conjecture related to the presence of flats in a CAT(0) space. Recall that the flat torus theorem allows to construct flats in a CAT(0) space $X$ from the data of an Abelian free subgroup of $\Isom(X)$. The flat closing conjecture asks for a kind of converse. Let $\Gamma$ be a group of isometries of $X$. A flat $F\subset X$ is \emph{$\Gamma$-periodic} if  there is a subgroup $\Gamma_0\leq \Gamma$ preserving $F$ such that $\Gamma_0\backslash X$ is compact. Observe that if $\Gamma$ is discrete, Theorem \ref{Bieberbach} implies that $\Gamma_0$ is virtually a free Abelian group of rank $\dim(F)$.

\begin{conjecture}[Flat Closing Conjecture  {\cite[6.B$_3$]{MR1253544}}]\index{flat closing!conjecture}\index{conjecture!flat closing} Assume $X$ is a proper CAT(0) space with a proper cocompact action of a discrete group $\Gamma$. Assume that $X$ contains a flat of dimension $n$. Is it true that $X$ contains a $\Gamma$-periodic flat of dimension $n$?
\end{conjecture}

Observe that a positive answer to that question would imply that $\Gamma$ contains $\Z^n$ as a subgroup, thanks to Theorem \ref{Bieberbach}. The conjecture is known to be true for $n=1$. Actually, the proof does not even use the existence of a geodesic. It only uses that some infinite discrete group acts geometrically on a proper CAT(0) \cite[Theorem 11]{MR1802725}.

The conjecture also holds for subgroups of discrete subgroups of isometry groups of symmetric spaces. Recall that the real rank of this isometry group coincides with the rank of its symmetric space.

\begin{proposition}[{\cite{MR0302822}}]Let $X$ be symmetric space and $\Gamma$ a discrete subgroup of $\Isom(X)$ acting cocompactly. Then $\Gamma$ contains a free Abelian subgroup $\Z^r$ where  $r=\rank(X)$.
\end{proposition}

\begin{proof} Let $G$ be the isometry group of $X$. This is a semi-simple Lie group with trivial center and no compact factor. The subgroup $\Gamma$ is a lattice in $G$ and by the Borel density theorem, the Zariski closure of $\Gamma$ is $G$ itself. The set of regular elements (elements with centralizer of minimal dimension) in $G$ is Zariski open and thus $\Gamma$ contains a regular element $\gamma$.  The centralizer $Z_G(\gamma)$ of $\gamma$ contains $Z_G(\gamma)\cap\Gamma$ as a lattice. Since $Z_G(\gamma)$ is virtually $\R^r$, the Bieberbach theorem (Theorem \ref{Bieberbach}) implies that $\Gamma$ contains a free Abelian group of rank $r$.
\end{proof}

One has to be aware that without properness, the conjecture is false. The Higman group acts cocompactly on a CAT(0) cube complex of dimension 2. This complex contains a lot of flats of dimension 2 but the Higman group does not contain any subgroup isomorphic to $\Z^2$ \cite[Proposition G]{Martin:2015qy}. Neither the space nor the action is proper. 
More generally, for manifolds, one has the following result in the special case where the flat is actually the Euclidean de Rham factor.

\begin{theorem}[{\cite{MR710052}}] Let $X$ be a simply-connected Riemannian manifold of non-positive curvature and $\Gamma$ a group acting freely and properly discontinuously on $X$. If $\Gamma\backslash X$ has finite volume then the rank of the maximal normal Abelian subgroup coincides with the dimension of the Euclidean de Rham factor.
\end{theorem}

Actually, it is shown that this maximal normal Abelian subgroup coincides with Clifford translations in $\Gamma$. Thus this theorem implies that the Euclidean de Rham factor is $\Gamma$-periodic, and that any subflat is so.
 
 For more general CAT(0) spaces, one has the following generalisation of Eberlein's result still dealing with $\Gamma$-periodicity of the Euclidean de Rham factor. 

\begin{theorem}[{\cite[Theorem 1.3(i)]{MR2574741}}]Let $X$ be a proper CAT(0) space such that $\Isom(X)$ acts minimally and cocompactly. If $\Gamma$ is a finitely generated lattice then $\Gamma$ virtually splits as $\Z^r\times\Gamma'$ where $r$ is the rank of the Euclidean de Rham factor.
\end{theorem}
\phantomsection
\section*{Acknowledgements}
\addcontentsline{toc}{section}{Acknowledgements}

I would like to thank Athanase Papadopoulos who proposed me to write this survey article. It was a great opportunity for me to try to circumnavigate the subject of groups acting on CAT(0) spaces and gather the important (in my mind and to the best of my knowledge) facts about this domain of group actions. I was happy to extract some structure from the many papers in this very active field since a few decades.

I would like to thank Indira Chatterji and Michah Sageev for answering a few questions about CAT(0) cube complexes.
\printindex
\bibliographystyle{habbrv}
\bibliography{../../Latex/Biblio/biblio.bib}
\end{document}